\documentclass[a4paper]{amsart}

\usepackage[utf8]{inputenc}
\usepackage[T1]{fontenc}
\usepackage{mathtools}
\usepackage{amssymb,amsfonts,amsmath}
\usepackage{comment}
\usepackage{enumitem}
\usepackage{needspace}
\usepackage{microtype}
\usepackage{url}
\usepackage{aliascnt}
\usepackage[hidelinks]{hyperref}
\usepackage[capitalise,nameinlink,noabbrev]{cleveref}

\usepackage{tikz-cd}
\usepackage{tikz}
\usetikzlibrary{arrows, automata, positioning}
\usetikzlibrary{decorations.pathreplacing}

\numberwithin{equation}{section}
\newtheorem{theorem}{Theorem}[section]
\newaliascnt{lemma}{theorem}
\newtheorem{lemma}[lemma]{Lemma}
\aliascntresetthe{lemma}
\newaliascnt{proposition}{theorem}
\newtheorem{proposition}[proposition]{Proposition}
\aliascntresetthe{proposition}
\newaliascnt{corollary}{theorem}
\newtheorem{corollary}[corollary]{Corollary}
\aliascntresetthe{corollary}
\theoremstyle{definition}
\newaliascnt{definition}{theorem}
\newtheorem{definition}[definition]{Definition}
\aliascntresetthe{definition}
\newaliascnt{notation}{theorem}
\newtheorem{notation}[notation]{Notation}
\aliascntresetthe{notation}
\newaliascnt{question}{theorem}
\newtheorem{question}[question]{Question}
\aliascntresetthe{question}
\crefname{question}{Question}{Questions}
\theoremstyle{remark}
\newaliascnt{remark}{theorem}
\newtheorem{remark}[remark]{Remark}
\aliascntresetthe{remark}
\crefname{theorem}{Theorem}{Theorems}
\crefname{lemma}{Lemma}{Lemmas}
\crefname{proposition}{Proposition}{Propositions}
\crefname{corollary}{Corollary}{Corollaries}
\crefname{definition}{Definition}{Definitions}
\crefname{notation}{Notation}{Notations}
\crefname{remark}{Remark}{Remarks}

\DeclareMathOperator{\Sym}{Sym}
\DeclareMathOperator{\EL}{EL}
\DeclareMathOperator{\MF}{MF}
\DeclareMathOperator{\Aut}{Aut}
\DeclareMathOperator{\CAT}{CAT}
\DeclareMathOperator{\GL}{GL}
\DeclareMathOperator{\FW}{FW}
\DeclareMathOperator{\FM}{FM}
\DeclareMathOperator{\SAut}{SAut}
\DeclareMathOperator{\SL}{SL}
\DeclareMathOperator{\Fix}{Fix}

\newcommand{\N}{\mathbb N}
\newcommand{\Z}{\mathbb Z}
\newcommand{\cT}{\mathcal T}

\newcommand{\frakC}{\mathfrak C}

\newcommand{\Qd}{\mathcal{Q}_d}
\newcommand{\Cstar}{\mathrm{C}^\ast}
\newcommand{\calK}{\mathcal K}

\newcommand{\Ad}{\mathrm{Ad}}
\newcommand{\Tr}{\mathrm{Tr}}
\newcommand{\Gp}{\mathcal{G}_p}
\newcommand{\Gq}{\mathcal{G}_q}

\title{Higher Brin--Thompson groups are Kazhdan}
\date{}
\hypersetup{pdftitle={Higher Brin--Thompson groups are Kazhdan}}

\author[R. Sauer]{Roman Sauer}
\address{Karlsruhe Institute of Technology, Englerstr.\ 2, 76131 Karlsruhe, Germany}
\email{roman.sauer@kit.edu}
\author[E. Schesler]{Eduard Schesler}
\address{Karlsruhe Institute of Technology, Englerstr.\ 2, 76131 Karlsruhe, Germany}
\email{eduardschesler@googlemail.com}
\makeatletter
\providecommand{\@wraptoccontribs}[2]{\wraptoccontribs{#1}{#2}}
\makeatother
\contrib[with an appendix by]{Francesco Fournier-Facio}
\address{Department of Mathematics, Heriot-Watt University, Edinburgh, UK}
\email{f.fournier-facio@hw.ac.uk}

\begin{document}
\begin{abstract}
We give an algebraic criterion for certain sequences $(G_n)_{n\in\N}$ of groups to have Kazhdan's property $(T)$ eventually. 
As a consequence, we prove that the Brin--Thompson groups $nV$ have Kazhdan's property $(T)$ for large~$n$.
The algebraic criterion provides a uniform proof that $nV$, $\Aut(F_n)$ and $\EL_n(R)$ for a finitely generated associative unital ring~$R$ have property $(T)$ for large~$n$. In an appendix, it is shown that the group $nV$ for large~$n$ is neither hyperlinear, in particular not sofic, nor $\MF$. Further, it is not Schatten $p$-approximable for any $1\le p<\infty$. 
\end{abstract}
\maketitle

\section{Introduction}

Ever since Kazhdan introduced property~$(T)$ in 1967 \cite{Kazhdan67}, it has remained a central notion in group theory and a fruitful source of new developments and applications.
In the same work, he established property~$(T)$ for certain simple algebraic groups of higher rank over local fields and their lattices, thereby providing the first non-compact groups with property $(T)$.
In subsequent work, Delaroche and Kirillov \cite{DelarocheKirillov68}, Vaserstein \cite{Vaserstein68}, and Wang \cite{Wang69} extended Kazhdan's results to all such higher-rank
groups, while Kostant \cite{Kostant69,Kostant75} exhibited groups of real rank one enjoying property~$(T)$.
For a long while, Lie and algebraic groups and their lattices remained the main sources of infinite discrete groups with property~$(T)$.
The foundations of a geometric and spectral approach were laid by Garland \cite{Garland73}, who connected spectral estimates on links of buildings to the vanishing of cohomology.
Cartwright, M{\l}otkowski, and Steger \cite{CartwrightMlotkowskiSteger94} subsequently established property~$(T)$ for certain groups acting on $\widetilde A_2$-buildings by a combinatorial argument independent of the representation theory of ambient algebraic groups.
Building on Garland's method, Żuk \cite{Zuk96} and Ballmann--Świątkowski \cite{BallmannSwiatkowski97} developed spectral criteria for property~$(T)$ in terms of links in simplicial complexes, yielding new examples beyond the realm of buildings.
These methods have since undergone substantial refinements and continue to evolve \cite{Zuk03,KotowskiKotowski13,Ashcroft23,Oppenheim26}.
Another route to proving property~$(T)$ without relying on suitable non-discrete ambient groups was opened by Shalom \cite{Shalom99}, who combined bounded generation with relative property~$(T)$ to obtain direct verifications of property $(T)$ for classical lattices.
Further development of the bounded generation method by Shalom and Vaserstein enabled them to establish property~$(T)$ for $\EL_n(R)$, $n\geq3$, over every finitely generated commutative unital ring $R$ \cite{Shalom06,Vaserstein07}.
These developments culminated in the work of Ershov and Jaikin-Zapirain \cite{EJ}, who combined Kassabov's relative property~$(T)$ results for matrix groups over non-commutative rings \cite{Kassabov07} with a spectral criterion inspired by Dymara and Januszkiewicz \cite{DymaraJanuszkiewicz02} to prove that $\EL_n(R)$ has property~$(T)$ for every finitely generated unital associative ring $R$ and every $n\geq3$.
This strengthening of Kazhdan's classical result for $\SL_n(\mathbb Z)$ to the setting of non-commutative rings had a natural parallel in the longstanding question of whether $\operatorname{Aut}(F_n)$ has property~$(T)$ in sufficiently large rank \cite[p.~128]{Lubotzky94}.
The groundwork for its eventual resolution was laid by Ozawa \cite{Ozawa16}, who showed that a finitely generated group $G$ has property~$(T)$ if and only if its Laplacian $\Delta$ satisfies $\Delta^2-\varepsilon\Delta=\sum_{i=1}^m\xi_i^*\xi_i$ for suitable $\varepsilon>0$ and $\xi_1,\ldots,\xi_m\in\mathbb{R}[G]$.
Building on this, Netzer and Thom \cite{NetzerThom15} developed a computational approach that uses semidefinite optimization to find approximate solutions to this equation and turns them into exact ones.
Further developments of this approach ultimately led to a positive resolution of the question for $\operatorname{Aut}(F_n)$ in every rank $n \geq 4$ \cite{KalubaNowakOzawa19,KalubaKielakNowak21,Nitsche22}.
More recently, Nitsche \cite{Nitsche26} obtained a computer-free proof of property~$(T)$ for $\operatorname{Aut}(F_n)$ for all sufficiently large $n$, using a limiting argument for normalized harmonic $1$-cocycles on $\operatorname{SAut}(F_n)$.

To this day, property~$(T)$ is often regarded as rare among naturally occurring groups.
However, this impression is challenged by the fact that known obstructions to property~$(T)$ are rather rare as well.
One such obstruction is the existence of an amenable action on a set without finite orbits, whose absence is known as property~$\FM$ \cite{CornulierFM}.
This already rules out property~$(T)$ for a wide range of naturally occurring groups, including all groups that virtually surject onto $\mathbb Z$.
A second obstruction is the existence of a Schreier graph with more than one end.
Sageev's construction \cite{Sageev95} turns such a Schreier graph in an unbounded action on a $\CAT(0)$ cube complex, whose existence rules out property~$(T)$ by a result of Niblo and Roller \cite{NibloRoller98}.
For finitely generated groups, the absence of such Schreier graphs is known as property~$\FW$.
Cornulier systematically studied this property \cite{CornulierFW} and showed that, even in conjunction with property~$\FM$, it does not imply property~$(T)$ \cite{CornulierFM}.
Nevertheless, it remains mysterious whether property~$(T)$ is genuinely rare among groups satisfying both~$\FM$ and~$\FW$.
In this paper, we explore this question within a certain family of Thompson-like groups, which have their origins in the groups $F$, $T$, and $V$, introduced by Thompson in 1965 \cite{Thompson65}.
While the infinite abelianization of $F$ immediately rules out property~$(T)$, the corresponding question for $T$ was an open problem of Ghys and Sergiescu \cite{GhysSergiescu87} that was answered in the negative by Reznikov \cite{Reznikov01}.
A far-reaching generalization of this result was obtained by Farley \cite{Farley03,FarleyCorrigendum,Farley05}, who constructed a proper action of $V$ on a $\CAT(0)$ cube complex, thereby establishing the Haagerup property for $V$ and ruling out property~$(T)$ for all its infinite subgroups.
Proper actions on $\CAT(0)$ cube complexes have since been constructed for many further families of Thompson-like groups, strengthening the impression that such actions are a common feature of these groups \cite{Hughes09,FarleyHughes17,BelkForrest19}.
Even in the class of Röver--Nekrashevych groups, which combine Thompson groups with self-similar groups, such as $\mathrm{SL}_n(\mathbb{Z})$~\cite{Kapovich12}, unbounded actions on $\CAT(0)$ cube complexes remain available \cite{WuWuZhaoZhou25,BashwingerDeBonis25}.
Alongside these geometric developments, Jones' construction of unitary representations of Thompson's groups, arising from his work on subfactor planar algebras \cite{Jones17}, has led to new proofs of the Haagerup property for $T$ and $V$ without relying on actions on $\CAT(0)$ cube complexes \cite{BrothierJones19,Brothier23}.
One family of Thompson-like groups that has resisted every attempt to establish the Haagerup property or construct an action on a $\CAT(0)$ cube complex with unbounded orbits is that of the Brin--Thompson groups $nV$, $n\geq 2$.
These groups were introduced by Brin \cite{Brin04} as a generalization of $V$ and consist of homeomorphisms of $\mathfrak{C}^n$, where $\mathfrak{C} =\{0,1\}^{\mathbb N}$ is the Cantor set.
More precisely, the elements of $nV$ are the homeomorphisms of $\mathfrak{C}^n$ that locally agree with elements of $V^n$, acting coordinate-wise.
A recent result of Callard and Salo rules out proper actions of $nV$ on $\CAT(0)$ cube complexes for $n\geq2$ by exhibiting distorted infinite cyclic subgroups \cite{CallardSalo24}.
Another proof of this fact was obtained even more recently by Kojima and Sheng by showing that $nV$ contains subgroups isomorphic to $\mathbb Q$ \cite{KojimaSheng26}.
However, it was still unknown whether the groups $nV$, $n\geq2$, have the Haagerup property~\cite{GenevoisQuestions,Tanner24,ZaremskyOpen}.
Similarly, the possibility that these groups might have property~$\FW$, or even property~$(T)$, remained open \cite{GenevoisQuestions,ZaremskyOpen,MedianGeometry26,FournierFacioWuZaremsky26}.
Regarding the question of property~$\FW$, the authors will give an affirmative answer for all $n\geq2$ in forthcoming work, building on a recent result of Matte Bon, Nekrashevych, and Zheng \cite[Theorem~3.1]{MBNZ2026Commensurating}.
Despite this, the non-rigid nature of the Brin--Thompson groups may make them seem unlikely candidates for property~$(T)$.
They therefore provide a particularly interesting test case for understanding the gap between Kazhdan groups and groups satisfying properties~$\FW$ and~$\FM$, although to the best of our knowledge, property $\FM$ was not yet established for Brin--Thompson groups.
Our main result aims to shed some light on this gap.

\begin{theorem}\label{thm:main-intro}
For sufficiently large~$n\in\N$, the Brin--Thompson group~$nV$ has Kazhdan's property~$(T)$. 
\end{theorem}

The proof of \cref{thm:main-intro} is inspired by Nitsche's asymptotic argument for $\Aut(F_n)$ \cite{Nitsche26}.
We recast and extend his argument in the language of affine isometric actions on Hilbert spaces.
Arguing by contradiction, we assume that property~$(T)$ fails for arbitrarily large $n$ and obtain a corresponding sequence of affine isometric actions without fixed points.
After choosing basepoints and suitably normalizing the actions, we pass to an ultralimit to obtain a limiting affine action whose basepoint is not a global fixed point.
We develop the argument further by studying weak limits of appropriate generators in this limiting action.
Under suitable algebraic hypotheses, these limits are mutually commuting orthogonal projections onto closed affine subspaces.
Using a barycentric argument, we then show that these projections fix the basepoint.
The final step is to obtain a contradiction by showing that the generators themselves fix the basepoint.
To this end, we exploit relations expressing certain products of two generators as conjugates of one of the factors.
The criterion for property~$(T)$ that we derive from this approach boils down to algebraic conditions on generators, their relations, and their behaviour under suitable automorphisms, see~\cref{thm:main}.
As we will see in \cref{sec:aut,sec:elementary}, the same approach also recovers, for all sufficiently large $n$, the property~$(T)$ results for $\Aut(F_n)$ \cite{KalubaNowakOzawa19,KalubaKielakNowak21,Nitsche22,Nitsche26} and $\EL_n(R)$ \cite{EJ}. We note that the resulting proof in the case of $\Aut(F_n)$ diverges from Nitsche's proof after \cref{lem:locality}, which is similar to~\cite[Lemma~2]{Nitsche26}.
We thus obtain the following corollary from a common algebraic framework.

\begin{corollary}\label{cor:applications-intro}
The following groups have property~$(T)$ for all sufficiently large $n$.
\begin{enumerate}[label=(\roman*)]
\item The Brin--Thompson groups $nV$.
\item The groups $\Aut(F_n)$.
\item The elementary groups $\EL_n(R)$, for every finitely generated associative unital ring $R$.
\end{enumerate}
\end{corollary}

In view of the preceding corollary, one could hope that our approach extends to mapping class groups of compact orientable surfaces.
However, this family of groups does not fit into the framework of $W$-groups, which are introduced in Section~\ref{sec:criterion} and form the basis of the proof of \cref{cor:applications-intro}.

In the world of Thompson-like groups it often happens that seemingly different constructions can yield isomorphic groups \cite{BleakDonovenJonusas17}.
For the Brin--Thompson groups, however, distinct values of $n$ yield non-isomorphic groups, as shown by Bleak and Lanoue \cite{BleakLanoue10}.
In fact, Matte Bon proved that $nV$ does not even embed into $mV$ when $n>m$ \cite[Corollary~1.11]{mattebon}.
Regarding this, the following question asks itself.

\begin{question}\label{que:T-for-2V}
Does the Brin--Thompson group $2V$ have property~$(T)$?
\end{question}

Our interest in property~$(T)$ for the Brin--Thompson groups was further fueled by recent developments in the theory of metric approximations of groups, particularly concerning soficity, hyperlinearity, and $\MF$ approximability.
In the case of soficity, OpenAI's criterion \cite[Proposition~2.3]{OpenAI26} can be easily applied to show that $nV$ cannot be sofic if it has property~$(T)$.
As pointed out by Francesco Fournier-Facio, recent developments concerning hyperlinearity and $\MF$ approximability \cite{Thom26,ALT26,caleb} allow one to obtain a stronger conclusion.
In the appendix, he combines these developments with \cref{thm:main-intro} to establish the following theorem, see~\cref{main:inapprox,schatten}.

\begin{theorem}\label{thm:inapproximability-intro}
For all sufficiently large $n$, the Brin--Thompson group $nV$ is neither sofic nor hyperlinear nor $\MF$.
Moreover, it is not Schatten $p$-approximable for any $1\leq p<\infty$.
\end{theorem}

Since the Brin--Thompson groups $nV$ are simple \cite{Brin10} and of type~$\mathrm{F}_{\infty}$ \cite{FMWZ13}, \cref{thm:inapproximability-intro} provides the first non-sofic (in particular, the first non-hyperlinear) examples, as well as the first non-$\MF$ examples, both among finitely presented simple groups and among groups of type~$\mathrm{F}_{\infty}$.
Moreover, to the best of our knowledge, they are the first naturally occurring
groups failing any of these three approximation properties,
having been introduced and extensively studied for reasons unrelated to approximation properties.
The combination of simplicity and finite presentability leads to a further application of \cref{thm:inapproximability-intro}.
For every sufficiently large $n$, we can specify a finite ball in $nV$ that admits none of the approximations considered in that theorem with arbitrarily small error.
Indeed, choose a finite presentation $nV=\langle S\mid R\rangle$ and let $r$ be the length of a longest relator in $R$.
Then $B_S(r)$ admits no $\varepsilon$-approximation for some $\varepsilon>0$ depending on the chosen notion of approximation.
Otherwise, approximations with errors tending to zero would yield a map from $B_S(r)$ into the corresponding metric ultraproduct that preserves partial multiplication and separates distinct elements.
The choice of $r$ ensures that the images of the generators satisfy all defining relations, giving a non-trivial homomorphism from $nV$ into this ultraproduct.
By simplicity, this homomorphism would be injective, contradicting the failure of the chosen approximation property.
The so obtained finite witnesses of inapproximability
suggest a tempting direction for future research, namely to understand their structure well enough to obtain obstructions to approximations without assuming property~$(T)$.

\subsection*{Statement on AI use}

\subsubsection*{Concerning the main text} The authors started this project a year ago with the initial goal to prove that the Brin--Thompson groups have the Haagerup property. After a few months of unsuccessful attempts we started to more and more believe that the Brin--Thompson groups have property~$(T)$. With the help of GPT-5.6 Sol we developed many computer experiments based on the sum-of-squares method. These computations turned out to be inconclusive. Then we turned to the paper of Nitsche~\cite{Nitsche26} who developed a computer-free approach to showing that $\Aut(F_n)$ has property~$(T)$ for large~$n$. We developed an axiomatic setting for this approach (see \cref{def:coordinates}) to widen its scope and started to develop combinatorial constraints to the affine action that result from the proof by contradiction. This process was based on our ideas but we did use GPT-5.6 Sol for background questions and as an interactive proof assistant. However, the combinatorial constraints we obtained were insufficient to reach a contradiction. Finally, we handed our results obtained thus far over to GPT-6 Astra. It contributed the key idea by suggesting to take weak limits of the resulting affine operators (cf. \cref{def:weak-limits}) and showed how to finish the argument. This helped to decrease the combinatorial complexity by turning combinatorial constraints into clean statements about affine actions on a Hilbert space.  Further, the authors used GPT-6 Astra in the process of integrating the proof for $nV$ into the axiomatic framework mentioned above and to make it into a uniform proof for all three classes of groups (cf. \cref{thm:main}). Finally, GPT-6 Astra was used for editorial assistance, including the creation of tikz diagrams. 
Whatever output the AI provided, was completely rewritten and reworked by the authors. 

\subsubsection*{Concerning the appendix} The free version of Claude assisted in drawing the tikz diagrams. No LLM was used at any point in the preparation of the appendix. Of course, much of what is discussed in the appendix is reliant on recent works that are heavily LLM-based.

\subsubsection*{Concerning the entire manuscript} The authors of the main text did a final proofreading run with GPT-6 Astra for the entire paper, including the appendix. The corrections were incorporated into the paper by all authors. 

\subsection*{Acknowledgements} We are grateful to  Matthew Brin, Adrien Le Boudec, Alex Lubotzky, Stefan Witzel, Matt Zaremsky for comments and feedback. R.S. and E.S. thank F.F.-F., the author of the appendix, for his careful reading of the main text and numerous detailed corrections. 

\section{The Kazhdan criterion}\label{sec:criterion}

In this section we discuss a Kazhdan criterion for increasing families of groups~$G_n$ that satisfy certain axioms. See \cref{thm:main}. The axioms are formulated for the colimit $G$ of these groups $G_n$. See \cref{def:coordinates}. Note that the group~$G$, which is used to prove property~$(T)$ for all but finitely many $G_n$ in \cref{thm:main}, does not have property~$(T)$ itself - for the simple reason that it is not finitely generated. 

To introduce the axiomatic framework, we define the following groups 
\[
  W=\Bigl(\bigoplus_{i\in\N}C_2\Bigr)\rtimes\Sym_{\mathrm{fin}}(\N),\qquad
  W_n=\Bigl(\bigoplus_{i=1}^n C_2\Bigr)\rtimes\Sym(\{1,\dots,n\}).
\]
Here $C_2$ is the group with two elements. 

We call the elements of the first factor \emph{sign changes} and the elements of the second factor \emph{permutations}. We use the usual cycle notation as in $(1,2)(3,4)$ for elements in $\Sym_{\mathrm{fin}}(\N)$. 
The sign change that is concentrated at $i$ is denoted by~$\varepsilon_i$.  
The \emph{support} of an element $(\epsilon, \sigma)$ of $W$ or $W_n$ is the union of the (finite) supports of $\epsilon$ and $\sigma$ in~$\N$. 
A group with an action of~$W$ by automorphisms is called a \emph{$W$-group}. 

\begin{definition}\label{def:coordinates}
A non-empty finite set $\cT\subseteq G$ of a $W$-group~$G$ is called a \emph{set of generator templates} of~$G$ if it has the following properties. 
\begin{enumerate}[label=(\roman*)]
\item Every element of $W$ supported outside $\{1,2\}$ fixes $\cT$ pointwise.
\item The set $\cT$ is closed under taking inverses and under the action of $W_2$. 
For every $t\in \cT$ we have $\varepsilon_1(t)=t^{-1}$ or $\varepsilon_2(t)=t^{-1}$. 
\item For $i\ne j$, denote by $t_{i,j}$ the image of $t\in\cT$ under any permutation that maps $1$ to $i$ and $2$ to $j$. For each unordered pair $p=\{i,j\}$, put
$\cT_p=\{t_{i,j}:t\in\cT\}$. We require that $\cT_p\cap\cT_q=\emptyset$ whenever $p\ne q$ and that $[s,t]=1$ for all $s\in \cT_p$ and $t\in \cT_q$ whenever $p\cap q=\emptyset$. 
\item The union of the sets $\cT_p$ over all unordered pairs $p\subset\N$ generates~$G$. 
\end{enumerate}
\end{definition}

Condition~(i) of \cref{def:coordinates} makes $t_{i,j}$ independent
of the chosen permutation.
Since the transposition $(1,2)$ preserves $\cT$, the set $\cT_p$ is also
independent of the ordering of $p$. 

\begin{notation}\label{not:coordinates}
For finite sets $I\subset\N$, define
\[
  G(I)=\langle\cT_p:p\subseteq I,\ |p|=2\rangle.
\]
Then $[G(I),G(J)]=1$ whenever $I\cap J=\varnothing$.
Further, we set \[G_n=G(\{1,\dots,n\})\quad\text{and}\quad S_n=\bigcup_{p\subset\{1,\dots,n\}}\cT_p.\]
\end{notation}

\begin{definition}\label{def:absorption}
A set of generator templates $\cT$ satisfies the \emph{absorption property}
if, for every $t\in\cT$, the element $t_{1,2}t_{1,3}$ is conjugate
to $t_{1,3}$ by an element of $G(\{2,3\})$.
\end{definition}

\begin{lemma}\label{lem:absorption-commutation}
Suppose that $\cT$ satisfies the absorption property.
For every $t\in\cT$ with $\varepsilon_1(t)=t^{-1}$, we have
\[
  [t_{1,2},t_{1,3}]=1.
\]
Consequently, every $t\in\cT$ satisfies
\[
  [t_{1,2},t_{1,3}]=1
  \quad\text{or}\quad
  [t_{2,1},t_{3,1}]=1.
\]
\end{lemma}
\begin{proof}
Put $X=t_{1,2}$ and $Y=t_{1,3}$, and choose
$\gamma\in G(\{2,3\})$ with $\gamma XY\gamma^{-1}=Y$.
If $\varepsilon_1(t)=t^{-1}$, then $\varepsilon_1$ inverts $X,Y$
and fixes $\gamma$. Applying it to the absorption identity and taking
inverses gives $\gamma YX\gamma^{-1}=Y$, whence $XY=YX$.
If instead $\varepsilon_2(t)=t^{-1}$, apply the first assertion to
the template $(1,2)\cdot t$, which is inverted by $\varepsilon_1$.
This gives the second alternative.
\end{proof}

\begin{definition}\label{def:local-commutators}
A set of generator templates $\cT$ satisfies the \emph{local commutator condition} if, for every $t,u\in\cT$, either $u=t^{\pm1}$ or there exist $\epsilon,\delta\in\{\pm1\}$ with
\[
[t_{1,2}^{\epsilon},u_{1,3}^{\delta}]
\in G(\{2,3,4\}).
\]
\end{definition}

Note that the commutator also lies in $G(\{1,2,3\})$ for obvious reasons. For the following definition note that for $t\in\cT$ we have $t=t_{1,2}$. 

\begin{definition}\label{def:local-conjugacy}
A set of generator templates $\cT$ satisfies
the \emph{local conjugacy condition} if $\cT$ consists of involutions and for every $t,u\in\cT$, 
there exist elements $s_{1,2},w_{1,2}$, conjugate to $t,u$, respectively, by
elements of $G(\{2,4\})$, satisfying the following conditions.
\begin{enumerate}[label=(\roman*)]
\item The products $s_{1,2}s_{1,3}$ and $w_{1,2}w_{1,3}$ are
conjugate to $s_{1,3}$ and $w_{1,3}$, respectively, by elements
of $G(\{2,3\})$.
\item The elements $s_{1,2}$ and $w_{1,3}$ commute.
\end{enumerate}
Here we put $s_{1,3}=(2,3)\cdot s_{1,2}$ and likewise for~$w$.
The elements $s_{1,2}$ and $w_{1,2}$ need not be elements of~$\cT$. 
They belong to $G(\{1,2,4\})$, and the auxiliary coordinate $4$
is suppressed in the notation.
\end{definition}

\cref{cor:applications-intro} in the introduction will be a consequence of the following general criterion. 

\begin{theorem}[Kazhdan criterion]\label{thm:main}
Let $\cT$ be a set of generator templates of a $W$-group $G$ that satisfies the absorption property. If, in addition, $\cT$ satisfies the local commutator condition or the local conjugacy condition, then $G_n$ has property~$(T)$ for all sufficiently large~$n$.
\end{theorem}

\section{A limiting affine action}\label{sec:limits}

The following result, which is our starting point, can be traced back 
to the harmonic-map existence theorem of Mok \cite{Mok95} and
Korevaar--Schoen \cite{KS97}. It can be found in \cite[Theorem~A.2 and (A.5)]{Kleiner10}. 

\begin{theorem}\label{thm:harmonic-basepoint}
Let $G$ be generated by a finite symmetric set $S$.
If $G$ does not have property~$(T)$, then $G$ admits an affine isometric
action on a real affine Hilbert space for which the energy
\[
  \mathcal E(x)=\frac1{|S|}\sum_{s\in S}d(x,sx)^2
\]
attains a positive minimum. Every minimizing point is the barycenter of its $S$-orbit. 
\end{theorem}

Throughout this section, we use only the generator-template assumptions of
\cref{def:coordinates}; neither the absorption property nor either local
condition is needed. We argue by contradiction and let $\mathcal N$ be an
infinite set of integers $n$ for which $G_n$ does not have property~$(T)$.
Put $M=|\cT|$. 

By \cref{thm:harmonic-basepoint}, for each $n\in\mathcal N$ there is
an affine isometric action of $G_n$ on a real affine Hilbert space $X_n$
for which we can and do choose a minimizing point~$o_n$ with respect to the energy for the generating set~$S_n$.

We induce this affine action to $G_n\rtimes W_n$ as in 
\cite[Section~2.5]{BHV08}. The induced affine Hilbert space~$H_n$ consists of the
maps $f:G_n\rtimes W_n\to X_n$ satisfying $f(kg)=g^{-1}f(k)$ for
all $k\in G_n\rtimes W_n$ and $g\in G_n$.
Evaluation at $w^{-1}$, $w\in W_n$, identifies the induced affine space with $X_n^{W_n}$,
equipped with its Hilbert product metric. The induced affine isometric action of $G_n\rtimes W_n$ is given by $(kf)(z)=f(k^{-1}z)$.
In this product model of~$H_n$, its restrictions to $G_n$ and $W_n$ are
\[
  (gy)_w=(w\cdot g)y_w,
  ~~\text{and}~~ (vy)_w=y_{wv}
  \quad\text{for $y\in H_n$, $g\in G_n$ and $v,w\in W_n$.}
\]
The diagonal basepoint $(o_n,\dots, o_n)$ is fixed by $W_n$ and is the barycenter of 
its $S_n$-orbit, since $W_n$ permutes $S_n$. The mean squared
displacement over $\cT\subset G_n$ at the diagonal basepoint is positive, since the basepoint in $X_n$ is not fixed by $S_n$ and $W_n\cdot\cT=S_n$. Thus, we can rescale 
the metric $d_n$ on $H_n$ such that 
\[
  \frac{1}{M}\sum_{t\in\cT}d_n(o_n,to_n)^2=1.
\]
Here, by a slight abuse of notation, we write $o_n$ for the diagonal basepoint $(o_n,\dots,o_n)$.

By the normalization and $W_n$-symmetry, every generator in $S_n$
displaces $o_n$ by at most $\sqrt M$.
In particular, a fixed word of length $L$ in the generators $t_{i,j}$
displaces $o_n$ by at most $L\sqrt M$ for all sufficiently large
$n\in\mathcal N$.

We regard each $H_n$ as a Hilbert space with origin $o_n$.
Choose a nonprincipal ultrafilter $\omega$ on $\mathcal N$ and take the
pointed ultralimit $(H_\infty, d_\infty, o_\infty)$ of $(H_n, d_n, o_n)$.
This is a Hilbert space with origin $o_\infty=[o_n]_\omega$.
Explicitly, $H_\infty$ consists of sequences $(y_n)$ at uniformly bounded distance
from $o_n$ identified when $\lim_\omega d_n(y_n,y'_n)=0$.

Since $W_n$ fixes $o_n$ and an element in $G_n$ of word length~$L$ displaces $o_n$ by at most $L\sqrt M$, every fixed element of $G\rtimes W$
displaces the basepoints by a uniformly bounded distance.  
Therefore 
\[
  g[y_n]_\omega=[gy_n]_\omega
\]
defines an affine isometric action of $G\rtimes W$. For each fixed
$g$, the right-hand side is defined for all sufficiently large $n$. 

We now restrict the action
to the closed affine span $H\subset H_\infty$ of the orbit $Go_\infty$. This span is also
$W$-invariant, since $W$ fixes $o_\infty$ and normalizes $G$. In the sequel we just write $d$ for the restriction of $d_\infty$ to $H$ and $o$ for $o_\infty$ as an element of $H$. The normalization above implies that 
\begin{equation}\label{eq:limit-normalization}
  \sum_{t\in\cT}d(o,to)^2=M.
\end{equation}

We recall Euler's quadrilateral law which is needed in the proof of \cref{lem:locality}. 
\begin{lemma}[Euler's quadrilateral law]\label{lem:square}
Let $g,h$ be commuting affine isometries of an affine Hilbert space. 
For any point $z$, let $m_1$ be the midpoint of the segment joining $z$ to
$ghz$, and let $m_2$ be the midpoint of the segment joining $gz$ to $hz$. Then
\begin{equation*}
  4d(m_1,m_2)^2
  =2d(z,gz)^2+2d(z,hz)^2-d(z,ghz)^2-d(gz,hz)^2.
\end{equation*}
\end{lemma}

\begin{proof}
Since $g,h$ commute, the opposite sides have equal lengths.
Thus the assertion follows from Euler's quadrilateral law, which is 
a generalization of the parallelogram law in case the quadrilateral $z,hz, ghz, gz$ is 
a parallelogram. 
\end{proof}

\begin{remark}
We consider the affine Hilbert spaces $H_n$ and $H$ as Hilbert spaces with respect to the origins
$o_n$ and $o$, respectively. The group actions are affine isometric and denoted by $gx$ or $g\cdot x$ for $g\in G$ and $x\in H$. We write $\ell(g)$ for the linear part of $g\in G$.
The entire $W$-action fixes $o$, so it is linear and orthogonal.
\end{remark}

The following result can be understood as a locality result saying that $G(I)$ and $G(J)$ do not interfere when $I\cap J=\emptyset$. It is inspired by~\cite[Lemma~2]{Nitsche26}. 

\begin{lemma}\label{lem:locality}
The affine $G$-action on $H$ has the following properties.
\begin{enumerate}[label=(\roman*)]
\item If $g\in G(I)$, $h\in G(J)$, and $I\cap J=\varnothing$, then
$o,go,gho, ho$ form a (possibly degenerate) rectangle. In particular,
\[
  \langle go, ho\rangle=0,
  \quad\text{and}\quad gho=go+ho.
\]
\item For $x\ne a$, let $q_{x,a}\in H$ be the barycenter of
$(t_{x,a}o)_{t\in\cT}$. For distinct $x,a,b$, we have
\[
  \langle q_{x,a},q_{x,b}\rangle\leq0.
\]
\end{enumerate}
\end{lemma}
\begin{proof}
Fix $n\in\mathcal N$ with $n\geq4$. For each pair
$p\subseteq\{1,\dots,n\}$, let $q_p\in H_n$ be the barycenter of
$(so_n)_{s\in\cT_p}$. The barycenter of the $q_p$ is $o_n$.
By symmetry, normalization and Cauchy--Schwarz,
\[
  \|q_p\|_n^2
  =\Bigl\|\frac1M\sum_{s\in\cT_p}so_n\Bigr\|_n^2
  \le\frac1M\sum_{s\in\cT_p}\|so_n\|_n^2=1.
\]
Taking the inner product of $\sum_p q_p=0$ with $q_{\{1,2\}}$ gives
$\sum_p\langle q_p,q_{\{1,2\}}\rangle_n=0$. Set
\[
  A_n=\langle q_{\{1,2\}},q_{\{3,4\}}\rangle_n,\qquad
  B_n=\langle q_{\{1,2\}},q_{\{1,3\}}\rangle_n,\qquad
  C_n=\|q_{\{1,2\}}\|_n^2.
\]
By permutation symmetry, $\langle q_p,q_{\{1,2\}}\rangle_n$ depends only
on $|p\cap\{1,2\}|$. Its value is $C_n$ for $p=\{1,2\}$,
$B_n$ for the $2(n-2)$ pairs meeting $\{1,2\}$ in one element,
and $A_n$ for the $\binom{n-2}{2}$ disjoint pairs. Thus
\begin{equation}\label{eq:count}
  0=C_n+2(n-2)B_n+\binom{n-2}{2}A_n.
\end{equation}
Their absolute values are at most one.
Let $g\in\cT_p$ and $h\in\cT_q$, where
$p,q\subseteq\{1,\dots,n\}$ are disjoint pairs. Writing $p=\{i,j\}$,
conditions~(i) and~(ii) of \cref{def:coordinates} ensure that one of
$\varepsilon_i,\varepsilon_j$ inverts $g$ and fixes $h$.
This sign change acts on $H_n$ by an isometry fixing $o_n$ and gives
$d_n(go_n,ho_n)=d_n(g^{-1}o_n,ho_n)=d_n(o_n,gho_n)$.
If $m_1,m_2$ are the diagonal midpoints of the quadrilateral
$o_n,go_n,gho_n,ho_n$, polarization and \cref{lem:square} give
\[
  \langle go_n,ho_n\rangle_n=d_n(m_1,m_2)^2\geq0.
\]
By bilinearity, $A_n$ is the mean of
these nonnegative terms for $g\in\cT_{\{1,2\}}$ and
$h\in\cT_{\{3,4\}}$. Equation~\eqref{eq:count} forces 
$B_n\leq0$ and $A_n=O(n^{-1})$. Each such term is bounded
by $M^2A_n$, also for any other disjoint pair of index sets by symmetry.
In the ultralimit, the diagonal midpoints must coincide, and we have 
\[
  d(go,ho)^2=d(o,go)^2+d(o,ho)^2.
\]
Hence the quadrilateral is a rectangle, proving~(i) when $g$ and $h$
are generators.
The inequality $B_n\leq0$ gives~(ii) in the limit, since finite
barycenters and inner products commute with the ultralimit.

Finally, we prove the rectangle property when the group elements are not necessarily generators. 
Fix a generator $h\in\cT_q$ with $q\subseteq J$.
For every generator $s$ supported on $I$, the rectangle property gives
\[
  \ell(s)ho=ho,\qquad so\perp ho.
\]
Hence $s$ preserves $(ho)^\perp$, and its linear part $\ell(s)$ 
fixes $ho$. Both properties are preserved under composition.
Hence, for every $g\in G(I)$,
\[
  go\perp ho,\qquad \ell(g)ho=ho.
\]
In particular, $gho=go+ho$. Since $g$ and $h$ commute, this also yields $\ell(h)go=go$.

Now fix an arbitrary $g\in G(I)$.
We have shown that every generator $h$ supported on $J$ satisfies
$\ell(h)go=go$ and $ho\perp go$.
Applying the same argument to these generators yields both
identities for every $h\in G(J)$. So we get 
\[
  go\perp ho,\qquad gho=hgo=go+ho,
\]
which proves the rectangle property for arbitrary $g,h$.
\end{proof}

\begin{remark}\label{rem: consequence about projections}
For finite $I\subset\N$, let $H_I$ be the closed linear span of
$G(I)o$ and let $\operatorname{pr}_I$ be the orthogonal projection
onto $H_I$. These subspaces have dense union. By
\cref{lem:locality}(i), disjoint index sets give orthogonal subspaces, and
\begin{equation}\label{eq:local-projection}
  \operatorname{pr}_I\circ h=\operatorname{pr}_I
  \quad\text{ for $h\in G(J)$ and $I\cap J=\varnothing$.}
\end{equation}
The linear part $\ell(h)$ fixes $H_I$ pointwise
and $ho$ is perpendicular to it. A permutation $\sigma$ fixing
$I$ pointwise also fixes $H_I$ pointwise. Similarly, it satisfies
$\operatorname{pr}_I\circ\sigma=\operatorname{pr}_I$.
\end{remark}

\section{The affine geometry at infinity}\label{sec:projections}

In this section we analyze in detail the affine $G$-action on the Hilbert space~$H$ we obtained as a limit in the previous section by assuming that there are infinitely many groups $G_n$ without property $(T)$ in the setting of \cref{thm:main}. 

We say that affine operators $A_j:H\to H$ on a Hilbert space~$H$ \emph{converge weakly} to~$A$ if,
for every $y\in H$, the points $A_j(y)$ converge weakly to $A(y)$ in $H$.
Equivalently, their linear parts $\ell(A_j)$ converge to $\ell(A)$ in the weak
operator topology, and the points $A_j(o)$ converge weakly to $A(o)$.
We call an affine operator \emph{contractive} if it is $1$-Lipschitz.

\begin{remark}\label{rem: convergence on dense subspace}
Let $H'\subset H$ be a dense linear subspace and let $A_j:H\to H$
be continuous affine operators. Suppose that
$\bigl(\langle A_jx,y\rangle\bigr)_j$ is Cauchy for every $x,y\in H'$
and that $\bigl(\|A_jz\|\bigr)_j$ is bounded for every $z\in H$.
Then $(A_j)_j$ converges weakly to an affine operator.

Set $C=\sup_j\|A_j(o)\|$. Since
$\ell(A_j)z=A_jz-A_j(o)$, the uniform boundedness principle~\cite[Theorem~2.6, p.~45]{Rudin91} implies that
$L=\sup_j\|\ell(A_j)\|<\infty$. For $x,y\in H$ and $x',y'\in H'$,
\begin{align*}
  |\langle A_jx-A_kx,y\rangle|
  &\le |\langle A_jx'-A_kx',y'\rangle|
    +\|A_jx'-A_kx'\|\|y-y'\|\\
  &\qquad\quad+\bigl(\|A_jx-A_jx'\|+\|A_kx-A_kx'\|\bigr)\|y\|\\
  &\le |\langle A_jx'-A_kx',y'\rangle|
    +2(C+L\|x'\|)\|y-y'\|\\
  &\qquad\quad+2L\|x-x'\|\|y\|.
\end{align*}
Choose $x',y'$ close to $x,y$, and then $j,k$ sufficiently large.
The right-hand side becomes arbitrarily small, so
$\bigl(\langle A_jx,y\rangle\bigr)_j$ is Cauchy for every $x,y\in H$.
By weak sequential completeness of $H$, the points $A_j(x)$ converge
weakly to a point $A(x)$ for every $x\in H$. The limit map is affine,
and its linear part has operator norm at most~$L$.
\end{remark}

\begin{lemma}\label{lem:weak-limits}
For every $x\in\N$ and $t\in\cT$, the affine operators $t_{x,a}$
converge weakly.
\end{lemma}

\begin{proof}
Fix $x\in\N$ and $t\in\cT$. By permutation symmetry and
\eqref{eq:limit-normalization}, for every $y\in H$ and $a\ne x$,
\begin{equation}\label{eq: displacement}
  \|t_{x,a}y\|=d(o,t_{x,a}y)
  \leq d(o,t_{x,a}o)+d(o,y)
  \leq\sqrt M+d(o,y).
\end{equation}
If $y,z\in H_I$ and $a,b\notin I\cup\{x\}$, the transposition
$(a,b)$ fixes $o$ and hence acts orthogonally. It fixes $y,z$
and conjugates $t_{x,a}$ to $t_{x,b}$. Hence
$\langle t_{x,a}y,z\rangle=\langle t_{x,b}y,z\rangle$.
Thus these coefficients are eventually constant for $y,z$ in the dense
linear subspace $\bigcup_I H_I$, and
\cref{rem: convergence on dense subspace} gives the conclusion.
\end{proof}

For the rest of this section, we fix $x=1$.
\begin{definition}\label{def:weak-limits}
For $t\in\cT$, let $E_t$ be the weak limit of $t_{1,a}$ as $a\to\infty$. The linear part of $E_t$ is denoted by $P_t$.
\end{definition}

\begin{remark}\label{rem:affine-projections}
A contractive linear idempotent map is a projection. Hence a contractive affine map
$Ay=Py+v$ with linear part~$P$ is an affine orthogonal projection, that is, the nearest point projection onto the image of~$A$, if $P$ is an idempotent and $Pv=0$.
\end{remark}

The idempotent property of $E_t$ is an easy and direct consequence of the absorption property when we let the parameters go to infinity. The fact that the idempotent $E_t$ is an affine orthogonal projection takes a bit more work. 

\begin{proposition}\label{prop:projection}
If $\cT$ satisfies the absorption property, then each $E_t$ is an
affine orthogonal projection.
\end{proposition}
\begin{proof}
Fix $t\in\cT$, and write $U_a=\ell(t_{1,a})$.
For distinct $a,b\ne 1$, the absorption property, upon conjugation with permutations, yields 
$\gamma_{a,b}\in G(\{a,b\})$ with
\begin{equation}\label{eq:template-absorption}
  \gamma_{a,b}t_{1,a}t_{1,b}\gamma_{a,b}^{-1}=t_{1,b}.
\end{equation}
Since the linear part of the conjugating element for distant $a,b$ acts trivially on a fixed test vector (cf.  \cref{rem: consequence about projections}), the linear parts of $t_{1,a}t_{1,b}$ and $t_{1,b}$ are locally indistinguishable. This is the idea which we now make precise. 
By equivariance, we may choose all $\gamma_{a,b}$ as images of one fixed
conjugator under permutations fixing $1$. In particular,
$c=\|\gamma_{a,b}o\|$ is independent of $a,b$. For finite $I$ and distinct
$a,b\notin I\cup\{1\}$, the linear part $\ell(\gamma_{a,b})$
fixes $H_I$ pointwise according to \cref{rem: consequence about projections}. 
Taking linear parts
in \eqref{eq:template-absorption} therefore gives
\[
  \langle v,U_aU_bw\rangle=\langle v,U_bw\rangle
\]
for all $v,w\in H_I$.
Take first $b\to\infty$ and then $a\to\infty$. Density gives
$P_t^2=P_t$. Since the maps $t_{1,a}$ are isometries, the limit map 
$E_t$ is contractive.  
By
\cref{rem:affine-projections}, $P_t$ is an orthogonal projection, and it remains to show that
$q:=P_tE_t(o)$ vanishes. 

Let us first describe the geometric idea behind the vanishing of~$q$, which is the component of the translation part of $E_t$ that is parallel to the subspace $\operatorname{Im} P_t$. Let us refer to this as the \emph{sliding component}. 
Now we approximate $E_t$ by two distant copies $t_{1,a}$ and $t_{1,b}$. The sliding components of their product approximately add up, so are approximately twice the one of $E_t$. The conjugating element in the absorption identity does not asymptotically affect the sliding component (another instance of \cref{rem: consequence about projections}). In the limit we obtain $2q=q$, thus $q=0$. 

To show that the sliding component vanishes, it suffices to test in the
direction of~$q$ itself. More precisely, we will prove that
\begin{equation}\label{eq: sliding component}
  \langle E_t(o),q\rangle=2\|q\|^2,
\end{equation}
whereas $q=P_tE_t(o)$ and the orthogonality of~$P_t$ give
$\langle E_t(o),q\rangle=\|q\|^2$. This will conclude the proof that $q=0$. 
Next we present the detailed argument for~\eqref{eq: sliding component}. 

Since $U_a$ converges weakly to $P_t$ and $P_tq=q$, we obtain that 
\[\Vert U_a^*q-q\Vert\to 0~~\text{as $a\to\infty$.}\] 
Since
$\ell(\gamma_{a,b})$ fixes $H_I$ whenever $a,b\notin I$ by \cref{rem: consequence about projections},
density also gives 
\[\ell(\gamma_{a,b})^*q\xrightarrow{a,b\to\infty} q.\]
Since the action is affine and $o$ is the origin,
we have $gy=\ell(g)y+go$ for $g\in G$ and $y\in H$.
We first show that the sliding components add asymptotically.
Since $U_a^*q\to q$ and the vectors $t_{1,b}o$ are uniformly bounded,
\[
  \langle t_{1,a}t_{1,b}o,q\rangle
  =\langle t_{1,a}o,q\rangle+\langle t_{1,b}o,U_a^*q\rangle\\
  \xrightarrow{a,b\to\infty}2\langle E_t(o),q\rangle
  =2\|q\|^2.
\]
Here we use that $\langle t_{1,b}o,(U_a^*-I)q\rangle\to0$ uniformly in $b$
by \eqref{eq: displacement}.

On the other hand, the absorption property identifies this product with a single copy
up to a conjugation whose effect in the direction~$q$ vanishes: 
Evaluating \eqref{eq:template-absorption} at $o$ gives
\[
  \ell(\gamma_{a,b})(t_{1,a}t_{1,b}o)+\gamma_{a,b}o
  =U_b(\gamma_{a,b}o)+t_{1,b}o.
\]
Hence
\[
  \langle t_{1,a}t_{1,b}o-t_{1,b}o,q\rangle
  =\langle t_{1,a}t_{1,b}o,(I-\ell(\gamma_{a,b})^*)q\rangle
       +\langle\gamma_{a,b}o,(U_b^*-I)q\rangle.
\]
By Cauchy--Schwarz,
\begin{equation*}
  |\langle t_{1,a}t_{1,b}o-t_{1,b}o,q\rangle|
  \le 2\sqrt M\,\|(I-\ell(\gamma_{a,b})^*)q\|
  +c\,\|(U_b^*-I)q\|\xrightarrow{a,b\to\infty} 0,
\end{equation*}
where $\|t_{1,a}t_{1,b}o\|\leq2\sqrt M$ by \eqref{eq: displacement}.
Finally, \eqref{eq: sliding component} follows from 
\[
  \langle E_t(o),q\rangle
  =\lim_{a,b\to\infty}\langle t_{1,a}t_{1,b}o,q\rangle
  =2\|q\|^2.\qedhere
\]
\end{proof}

A major step in the proof is \cref{prop:commuting-limits} saying that $E_t$ and $E_u$ commute for $t,u\in\cT$. We first deal with the special case $u=t^{-1}$ in which case the two projections coincide. 

\begin{lemma}\label{lem:inverse-limits}
For every $t\in\cT$, we have $E_{t^{-1}}=E_t$.
\end{lemma}
\begin{proof}
The linear parts agree, since $P_{t^{-1}}=P_t^*=P_t$.
Suppose first that the sign change map $\varepsilon_2$ (see \cref{def:coordinates}) satisfies $\varepsilon_2(t)=t^{-1}$.
For $g\in G(I)$ and $a\notin I\cup\{1\}$, choose a permutation
$\sigma$ with $\sigma(1)=1$ and $\sigma(2)=a$.
Then $t_{1,a}=\sigma(t)$ and $\varepsilon_a=\sigma\varepsilon_2\sigma^{-1}$, so
$  \varepsilon_a(t_{1,a})
  =(\sigma\varepsilon_2\sigma^{-1})(\sigma(t))
  =\sigma(\varepsilon_2(t))
  =\sigma(t^{-1})=t_{1,a}^{-1}.$
Since $\varepsilon_a$ acts orthogonally and fixes both $o$ and $go$, we obtain
\[
  \bigl\langle t_{1,a}^{-1}o,go\bigr\rangle
  =\bigl\langle \varepsilon_a(t_{1,a}o),\varepsilon_a(go)\bigr\rangle
  =\bigl\langle t_{1,a}o,go\bigr\rangle.
\]
Letting $a\to\infty$ shows
\[
  \bigl\langle E_{t^{-1}}(o)-E_t(o),go\bigr\rangle=0.
\]
Since the $G$-orbit of $o$ has dense linear span
in $H$, we get $E_{t^{-1}}(o)=E_t(o)$. Since their linear parts also agree,
the affine projections agree.

Otherwise $\varepsilon_1(t)=t^{-1}$. \cref{lem:absorption-commutation} and permutation symmetry give
$[t_{1,a},t_{1,b}]=1$.
Hence $[t_{1,a},t_{1,b}^{-1}]=1$. Taking first $b\to\infty$ and
then $a\to\infty$ shows that $E_t$ and $E_{t^{-1}}$ commute. 
Moreover, they are orthogonal projections onto parallel affine subspaces since their linear parts coincide. If they commute, they must be the same orthogonal projection. 
\end{proof}

The following lemma is standard. We leave it to the reader. 
For the linear parts, one uses that weak operator convergence
of orthogonal projections to an orthogonal projection implies strong
operator convergence~\cite[\S115, p.~62]{Halmos82}, and that multiplication
is jointly continuous on norm-bounded sets in the strong operator topology.
In the affine setting, one additionally needs to control the translation
parts. 

\begin{lemma}\label{lem:affine-projection-limits}
Let $H$ be a real Hilbert space with origin $o$.
\begin{enumerate}[label=(\roman*)]
\item Let $E,F_n$ be affine orthogonal projections satisfying
$EF_n=F_nE$. If $\ell(F_n)\to\ell(E)$ in the weak
operator topology and $(F_n(o))$ is bounded, 
then $F_n\to E$ converges weakly.
\item Let $F_n,G_n$ be affine orthogonal projections converging
weakly to affine orthogonal projections $F,G$, respectively.
If $F_nG_n=G_nF_n$ for every $n$, then $FG=GF$.
\end{enumerate}
\end{lemma}

\begin{proposition}\label{prop:commuting-limits}
If the local commutator condition or the local conjugacy condition
holds, then the projections $E_t$, $t\in\cT$, commute pairwise.
\end{proposition}
\begin{proof}[Proof under the local commutator condition]
Let $t,u\in\cT$. If $u=t^{\pm1}$, then $E_u=E_t$ by
\cref{lem:inverse-limits}. Otherwise, by the same lemma, replacing
either template by its inverse does not change its associated projection.
Thus, by \cref{def:local-commutators}, we may assume that
$[t_{1,2},u_{1,3}]\in G(\{2,3,4\})$.
Now fix a finite set $I$ and choose distinct
$a,b,c\notin I\cup\{1\}$. Apply a permutation fixing $1$ and sending
$2,3,4$ to $a,b,c$ to this inclusion.
The commutator $d=[t_{1,a},u_{1,b}]$ belongs to $G(\{a,b,c\})$. When $I$ increases, the failure for $t_{1,a}$ and $u_{1,b}$ to commute is supported on indices escaping to infinity and thus invisible in the weak limit. More precisely, we have 
\[
  \operatorname{pr}_I(t_{1,a}u_{1,b}y)
  =\operatorname{pr}_I(u_{1,b}t_{1,a}y)
\]
for every $y\in H$ according to~\eqref{eq:local-projection}.
Take first $b\to\infty$, then $a\to\infty$. Since $I$ was
arbitrary and the union of the images of the $\operatorname{pr}_I$ is dense in~$H$, this shows 
$E_tE_u=E_uE_t$.
\end{proof}

\begin{proof}[Proof under the local conjugacy condition]
Fix $t,u\in\cT$. The idea is to replace the template generators $t,u$ by suitable local
conjugates whose limits commute. In the limit process we fix another auxiliary coordinate~$c$ other than the base coordinate~$1$. We call the limit operators $F_c, H_c$ as they depend on the auxiliary coordinate. 
First, we show that $F_c,H_c$ are affine orthogonal projections and,
by taking successive limits as in the preceding proof, that they commute.
Second, we recover $E_t,E_u$ as limits of $F_c, H_c$ as 
the auxiliary coordinate $c$ tends to infinity. 
Third, we use that commutativity is preserved under weak limits of
affine orthogonal projections.

\subsubsection*{Step 1: Construct commuting projections with $c$ fixed}

Choose $s_{1,2},w_{1,2}$ as in \cref{def:local-conjugacy}.
Both belong to $G(\{1,2,4\})$.
Fix an auxiliary index $c\ne 1$. For $a\notin\{1,c\}$, denote
their images under a permutation fixing $1$ and sending $2,4$ to $a,c$ by
$s_a,w_a$. These are well-defined and belong to
$G(\{1,a,c\})$. With $c$ fixed, the argument of
\cref{lem:weak-limits} shows that the weak limits of the affine maps $s_a$ and $w_a$ as $a\to\infty$ exist. We will denote them by $F_c$ and $H_c$. 

We next verify that $F_c$ and $H_c$ are affine orthogonal projections.
By condition~(i), there is an element $\gamma\in G(\{2,3\})$ such that
\[
  \gamma s_{1,2}s_{1,3}\gamma^{-1}=s_{1,3}.
\]
For distinct $a,b\notin\{1,c\}$, applying a permutation fixing $1$ and sending
$2,3,4$ to $a,b,c$, respectively, gives
\[
  \gamma_{a,b}s_as_b\gamma_{a,b}^{-1}=s_b,
  \qquad \gamma_{a,b}\in G(\{a,b\}).
\]
Thus the family $(s_a)_a$, with $c$ fixed, satisfies the same absorption
relation~\eqref{eq:template-absorption} used in the proof of \cref{prop:projection}. In particular,
the conjugators involve only the running indices $a,b$, not the auxiliary
index $c$. 
The absorption relation and the locality of the conjugators are therefore
precisely those used in the proof of \cref{prop:projection}.
The remaining estimates require only uniform bounds on
$\|\gamma_{a,b}o\|$ and $\|s_a o\|$.
These hold because both families consist of permutation images of fixed
elements and permutations fix $o$.
Thus the same proof shows that $F_c$ and $H_c$ are affine orthogonal projections.

Finally, condition~(ii) in \cref{def:local-conjugacy}, transported by the same permutations, yields 
\[
  s_aw_b=w_bs_a\qquad\text{where $a\ne b$, $a,b\notin\{1,c\}$.}
\]
With $c$ fixed, taking weak limits first as $b\to\infty$ and then as
$a\to\infty$ gives $s_aH_c=H_cs_a$ and then $F_cH_c=H_cF_c$.

\subsubsection*{Step 2: Recover the original projections as $c\to\infty$}

We show that $F_c$ converges weakly to $E_t$. Likewise for $H_c$ and $E_u$.
Here we use the assumption in \cref{def:local-conjugacy} that $\cT$
consists of involutions. For distinct $a,b\notin\{1,c\}$, absorption shows
that $t_{1,a}t_{1,b}$ is conjugate
to $t_{1,b}$, so $(t_{1,a}t_{1,b})^2=1$.
Since both factors are involutions, this implies $[t_{1,a},t_{1,b}]=1$.
Since $s_b$ is conjugate to $t_{1,b}$
by an element of $G(\{b,c\})$, it also commutes with $t_{1,a}$.
Hence $E_tF_c=F_cE_t$, and similarly $E_uH_c=H_cE_u$.

We now consider the linear part $Q_c$ of $F_c$. For finite $I$ and distinct
$a,c\notin I\cup\{1\}$, the linear part of the conjugator between $t_{1,a}$
and $s_a$ fixes $H_I$ pointwise.
Thus $\ell(s_a)$ and $\ell(t_{1,a})$ have the same matrix
coefficients on this space. Taking first $a\to\infty$ and then
$c\to\infty$, and using density, gives
$Q_c\to P_t$ weakly. Likewise, the linear parts $R_c$ of $H_c$
converge weakly to $P_u$.
The points $F_c(o)$ remain bounded, since every $s_a$, as $a,c$ vary,
is a permutation image of the fixed element $s_{1,2}$.
The same argument bounds $H_c(o)$ using $w_{1,2}$. 
Part~(i) of \cref{lem:affine-projection-limits} implies 
$F_c\to E_t$ and $H_c\to E_u$ weakly.

\subsubsection*{Step 3: Pass commutativity to the limit $c\to\infty$}

By Step~1, $F_cH_c=H_cF_c$ for every $c$.
By Step~2, $F_c\to E_t$ and $H_c\to E_u$ weakly.
Part~(ii) of \cref{lem:affine-projection-limits} implies
$E_tE_u=E_uE_t$.
\end{proof}

Every affine projection $E_t$, $t\in \cT$, yields a reflection at the affine subspace fixed by $E_t$. By \cref{prop:commuting-limits}, the projections commute, and so do these reflections, which act as a finite product of cyclic subgroups of order~$2$ on the Hilbert space~$H$. By a well-known fixed point theorem for isometric actions of finite groups on complete CAT(0)-spaces~\cite[Corollary II.2.8(1) on p.~179]{Bridson-Haefliger}, which goes back to Bruhat and Tits, there is a common fixed point of the $E_t$, $t\in\cT$. Below we show that the origin $o$ is such a common fixed point. After that, we deduce that the elements $t\in\cT$ themselves fix~$o$, which requires the absorption property. 

\begin{proposition}\label{prop:final}
If $\cT$ satisfies the absorption property and the projections
$E_t$, $t\in\cT$, commute pairwise, then every $t\in\cT$ fixes $o$.
\end{proposition}
\begin{proof}
We consider the barycenters $q_{1,a}$ from \cref{lem:locality}. 
These points converge weakly to $m=M^{-1}\sum_{t\in\cT}E_t(o)$.
Taking first $a\to\infty$ and then $b\to\infty$ in
\cref{lem:locality}(ii) implies that 
\[
  \|m\|^2
  =\lim_{b\to\infty}\lim_{a\to\infty}
    \langle q_{1,a},q_{1,b}\rangle\leq0.
\]
Hence $m=o$, so $o$ is the barycenter of
$(E_t(o))_{t\in\cT}$.

By the preceding observation, the projections $E_t$ have a common fixed
point~$q$. Since each $E_t$ is an orthogonal projection,
$E_t(o)\perp q-E_t(o)$. Since $o$ is the barycenter of
$(E_t(o))_{t\in\cT}$, we obtain
\[
  \sum_{t\in\cT}\|E_t(o)\|^2
  =\Bigl\langle q,\sum_{t\in\cT}E_t(o)\Bigr\rangle=0.
\]
Thus $E_t(o)=o$ for every $t\in\cT$.

For finite $I\subset\N$, permutations fixing $I\cup\{1\}$ pointwise
also fix $H_I$ pointwise.
Hence $\operatorname{pr}_I(t_{1,a}o)$ is independent of
$a\notin I\cup\{1\}$. By weak continuity of the projection,
its limit is $\operatorname{pr}_I(E_t(o))=o$.
So $t_{1,a}o\perp H_I$ for every $a\notin I\cup\{1\}$.
By permutation symmetry, applied also to the index set $I$, and the
independence of $\cT_p$ from the ordering of $p$, we obtain
\begin{equation}\label{eq: perpendicular}
  so\perp H_I
  \qquad\text{for }s\in\cT_p\text{ with }p\not\subseteq I.
\end{equation}

Fix $t\in\cT$. By \cref{lem:absorption-commutation}, one of the
templates $u=t$ or $u=(1,2)\cdot t$ satisfies $[u_{1,2},u_{1,3}]=1$.
Put $X=u_{1,2}$ and $Y=u_{1,3}$. By the absorption property there is an element
$K\in G(\{2,3\})$ with $KYK^{-1}=XY$.
From~\eqref{eq: perpendicular} we get $X^{\pm1}o\perp Yo$.
Pythagoras' theorem gives 
\[
  d(Xo,Yo)^2=d(o,Xo)^2+d(o,Yo)^2.
\]
Similarly, 
\[
d(o,XYo)^2=d(X^{-1}o, Yo)^2=d(o, X^{-1}o)^2+d(o, Yo)^2=d(o,Xo)^2+d(o,Yo)^2.
\]
By \cref{lem:square}, the diagonal midpoints coincide, so
$o,Xo,XYo,Yo$ form a rectangle. In particular, $XYo=Xo+Yo$.
As before, \eqref{eq: perpendicular} yields 
\[Ko\perp Xo,Yo\quad\text{and}\quad K^{-1}o\perp Y^{-1}o.\]
See \cref{fig:pythagoras-final} for schematic overview of the situation. 
Next we apply Pythagoras, invariance of the distance and, specifically for the second equality sign, the equation $(XY)^{-1}K=KY^{-1}$ to obtain 
\begin{align*}
  \|Xo\|^2+\|Yo\|^2+\|Ko\|^2
  =d(XYo,Ko)^2
  &=d(Y^{-1}o,K^{-1}o)^2\\
  &=\|Yo\|^2+\|Ko\|^2.
\end{align*}

\begin{figure}[htbp]\label{fig:pythagoras-final}
\centering
\begin{tikzpicture}[scale=0.9, every node/.style={font=\small}]
  \coordinate (O) at (0,0);
  \coordinate (X) at (2,0);
  \coordinate (Y) at (.8,.7);
  \coordinate (Z) at (2.8,.7);
  \coordinate (K) at (0,2);
  \fill[gray!8] (O)--(X)--(Z)--(Y)--cycle;
  \draw[thin] (O)--(X)--(Z)--(Y)--cycle;
  \draw[thick] (O)--(Z)--(K)--cycle;
  \draw[thin] (1.8,0)--(1.96,.14)--(2.16,.14);
  \draw[thin] (0,.2)--(.28,.27)--(.28,.07);
  \node[below left] at (O) {$o$};
  \node[below] at (X) {$Xo$};
  \node[above left] at (Y) {$Yo$};
  \node[right] at (Z) {$XYo$};
  \node[above] at (K) {$Ko$};
  \node at (1.5,-.65) {$Ko\perp\operatorname{span}\{Xo,Yo\}$};
  \begin{scope}[xshift=5.4cm]
    \draw[thick] (0,0)--(2.4,0)--(0,2)--cycle;
    \draw[thin] (.2,0)--(.2,.2)--(0,.2);
    \node[below left] at (0,0) {$o$};
    \node[below] at (2.4,0) {$Y^{-1}o$};
    \node[above] at (0,2) {$K^{-1}o$};
  \end{scope}
\end{tikzpicture}

\caption{Schematic view of the geometric configuration.}
\label{fig:pythagoras-final}
\end{figure}
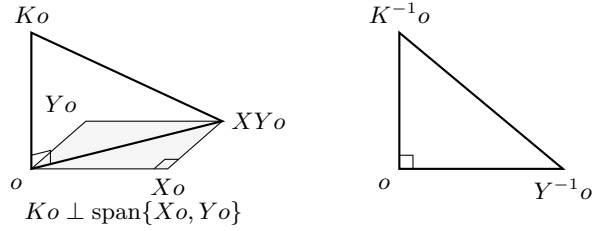

Thus $Xo=o$, hence $uo=o$. Since $W$ fixes $o$, the transposition
$(1,2)$ gives $to=o$ in either case.
\end{proof}

\begin{proof}[Proof of \cref{thm:main}]
The proof is by contradiction. By assuming that there are infinitely many groups $G_n$ without property $(T)$ we construct in \cref{sec:limits} an affine action of $G$, the colimit of the groups $G_n$ without property~$(T)$, on a Hilbert space~$H$. 
By taking weak limits of elements in $G$ we construct affine operators $E_t$ on $H$ (see \cref{def:weak-limits}) which are affine orthogonal projections (\cref{prop:projection}) by the absorption property. 
By \cref{prop:commuting-limits}, the projections $E_t$ commute
pairwise. For that we require either the local conjugacy property or the local commutator property. Again by the absorption property and \cref{prop:final} we obtain that every $t\in\cT$ fixes $o$. This achieves the contradiction via \eqref{eq:limit-normalization}. So we proved that all but finitely many groups $G_n$ have property~$(T)$. 
\end{proof}

\section{Applications}\label{sec:applications}

\subsection{Higher Brin--Thompson groups}\label{sec:brin}

We write $\frakC=\{0,1\}^{\N}$ for the Cantor set of infinite binary sequences and let $\{0,1\}^{*}$ denote the set of finite binary words, including the empty word.
The Brin--Thompson group $nV$ is the group of homeomorphisms of $\frakC^n$ that are locally given by the coordinate-wise action of elements of $V^n$, where $V$ is Thompson's group with its standard action on $\frakC$.
To describe these homeomorphisms more concretely, we associate to each $w \in \{0,1\}^{*}$ the set $C(w)=w\frakC$ of infinite binary sequences starting with $w$, which we call the \emph{cylinder} corresponding to $w$.
A product $C(w_1)\times\cdots\times C(w_n)$ of such cylinders is referred to as a \emph{dyadic rectangle} in $\frakC^n$ and will be identified with its \emph{address} $\boldsymbol{w}=(w_1,\ldots,w_n)$.
For two
addresses $\boldsymbol{v}=(v_1,\ldots,v_n)$ and $\boldsymbol{w}=(w_1,\ldots,w_n)$, the \emph{prefix replacement} from $\boldsymbol{v}$ to $\boldsymbol{w}$ is the homeomorphism
given by
\[
(v_1\xi_1,\ldots,v_n\xi_n)
\longmapsto
(w_1\xi_1,\ldots,w_n\xi_n),
\qquad (\xi_1,\ldots,\xi_n)\in\frakC^n.
\]
With this terminology, a homeomorphism $g$ of $\frakC^n$ belongs to $nV$ if and only if there is a finite partition $\frakC^n=\bigsqcup_{i=1}^k R_i$ into dyadic rectangles such that each restriction $g|_{R_i}$ is a prefix replacement onto a dyadic rectangle.
Since $g$ is a homeomorphism, the image rectangles $g(R_1),\ldots,g(R_k)$ also form a partition of $\frakC^n$.
The generator templates that we will construct in order to apply \cref{thm:main} will be chosen from among the \emph{address transpositions}, defined as follows.
For two addresses $\boldsymbol{v}$ and $\boldsymbol{w}$ representing disjoint dyadic rectangles in $\frakC^n$, the address transposition $(\boldsymbol{v}\ \boldsymbol{w})\in nV$ is the homeomorphism that interchanges these rectangles by prefix replacement and acts as the identity on their complement.
Note that we can embed $nV$ into $(n+1)V$ by letting it act trivially on the additional coordinate.
Using these embeddings, we define the group $G = \bigcup_{n\geq2}nV$, which we think of as a group of homeomorphisms of the infinite Cartesian product $\frakC^{\N}$ in the canonical way.
Recall that $W=\bigl(\bigoplus_{i\in\N}C_2\bigr)\rtimes\Sym_{\mathrm{fin}}(\N)$.
The action of $\Sym_{\mathrm{fin}}(\N)$ on $\frakC^{\N}$ by coordinate permutations induces an action on $G$ by conjugation.
By letting all sign changes act trivially, we extend this to an action of $W$ and thus regard $G$ as a $W$-group.
Our generator templates for this $W$-group lie in $2V$, regarded as the subgroup of $G$ acting on the first two coordinates.
This allows us to describe them using address transpositions with $n=2$.
Put
\[
  \Lambda=\{00,01,10,11\},\qquad
  \mathcal E=\{(00,01),(01,10),(10,11)\},
\]
and define
\begin{align*}
 \cT_{\mathrm{grid}}
  ={}&\{((p,r)\ (q,r))\mid(p,q)\in\mathcal E,\ r\in\Lambda\}\\
  &{}\cup\{((r,p)\ (r,q))\mid(p,q)\in\mathcal E,\ r\in\Lambda\}.
\end{align*}
For $\varepsilon\in\{0,1\}$, let $z_\varepsilon=((00,00)\ (10\varepsilon,00))$.
We will consider the following set of generator templates
\[
  \cT=\cT_{\mathrm{grid}}\cup
  \{z_0,z_1,(1,2)\cdot z_0,(1,2)\cdot z_1\}.
\]
By construction, $\cT$ is invariant under $(1,2)$ and consists of $28$ involutions.

\begin{lemma}\label{lem:brin-generation}
For every finite set $I\subset\N$ with $|I|\geq2$, the group $G(I)$
is the standard copy of $|I|V$ on the coordinates indexed by $I$.
In particular, $G_n=nV$.
\end{lemma}
\begin{proof}
Consider the graph with vertex set $\Lambda^2$ in which two vertices are joined whenever their transposition belongs to $\cT_{\mathrm{grid}}$.
This graph is a $4\times4$-grid and hence connected.
Since the edge transpositions of a connected graph generate the full symmetric group on its vertex set, we obtain $\langle\cT_{\mathrm{grid}}\rangle=\Sym(\Lambda^2)$, where we view $\Sym(\Lambda^2)$ as the subgroup of $2V$ whose elements permute the rectangles with addresses in $\Lambda^2$ by prefix replacement.
In particular we see that the element $u=((10,00)\ (00,01))$ belongs to $\langle\cT_{\mathrm{grid}}\rangle$.
By conjugating $z_\varepsilon$ with $u$ we obtain
\[
  uz_\varepsilon u^{-1}
   =((00,00)\ (00\varepsilon,01))\qquad(\varepsilon=0,1).
\]
We now use these elements to show that $\langle\cT\rangle$ contains a three-element generating set $\{a,b,c\}$ of $2V$ provided by Quick \cite[Theorem~1.2 and Section~3]{Quick}.
In his construction, $a$ and $b$ generate $\Sym(\Lambda^2)$, so both belong to $\langle\cT\rangle$.
To obtain $c$, we choose the first three addresses in Quick's enumeration of $\Lambda^2$ to be 
$(00,00)$, $(00,01)$ and $(00,10)$.
Then his relation~R7 gives us
\[
  c=\bigl((1,2)\cdot z_1\bigr)\bigl((1,2)\cdot z_0\bigr)
    (uz_1u^{-1})(uz_0u^{-1}) \in \langle\cT\rangle
\]
and hence $\langle\cT\rangle = 2V$.
Thus, by permuting coordinates, we see that $\langle\cT_p\rangle$ is the standard copy of $2V$ on the coordinates in $p$ for every two-element subset $p\subseteq I$.
Applying \cite[Proposition~3.2]{BZ} to the complete graph on $I$, we conclude that $G(I)$ is the standard copy of $|I|V$ on the coordinates in $I$.
\end{proof}


\begin{lemma}\label{lem:brin-templates}
The set $\cT$ is a set of generator templates of the $W$-group $G$.
\end{lemma}
\begin{proof}
Every template lies in the standard copy of $2V$ on the first two coordinates.
Hence every permutation that fixes $1$ and $2$ fixes $\cT$ pointwise.
Since every template is an involution, $\cT$ is closed under taking inverses.
By definition, all sign changes act trivially, so $\varepsilon_1(t)=\varepsilon_2(t)=t=t^{-1}$ for every $t\in\cT$.
Together with the invariance of $\cT$ under $(1,2)$, this also shows that $\cT$ is preserved by $W_2$.
If $p\cap q=\varnothing$, then the elements of $\cT_p$ and $\cT_q$ act on disjoint coordinate pairs and hence commute.
Observe next that, by construction of $\cT$, for $p = \{i,j\}$, each $t\in\cT_p$ fixes all coordinates except one in $p$, say $i$.
In coordinate $i$, the element $t$ applies a fixed prefix transposition when coordinate $j$ lies in a prescribed non-empty proper cylinder, and acts as the identity otherwise.
Thus $j$ is the unique coordinate other than $i$ whose value affects the action on coordinate $i$.
Both $i$ and $j$ are therefore determined by $t$, so $\cT_p\cap\cT_q=\varnothing$ whenever $p\ne q$.
Finally, \cref{lem:brin-generation} gives us $G_n=nV$ for every $n\geq2$, so the sets $\cT_p$ together generate $G=\bigcup_{n\geq2}nV$.
\end{proof}

To verify the absorption property and the local conjugacy condition, we recall two well-known facts about clopen partitions and involutions in $nV$.

\begin{lemma}\label{lem:brin-matching}
For any two finite partitions
\[
  \frakC^n=\bigsqcup_{i=1}^m A_i=\bigsqcup_{i=1}^m B_i
\]
into non-empty clopen sets, there exists $g\in nV$ such that $g(A_i)=B_i$ for every $i=1,\ldots,m$.
Moreover, any two non-trivial involutions in $nV$ with non-empty fixed sets are conjugate in $nV$.
\end{lemma}
\begin{proof}
Every clopen subset of $\frakC^n$ is a finite disjoint union of dyadic rectangles.
For each pair $(A_i,B_i)$, choose such decompositions of $A_i$ and $B_i$ and refine them so that they contain the same number of rectangles.
For each $i$, choose a bijection between the rectangles in the decomposition of $A_i$ and those in the decomposition of $B_i$.
The corresponding prefix replacements together define an element $g\in nV$ satisfying $g(A_i)=B_i$ for every $i=1,\ldots,m$.
For the second assertion we apply \cite[Proposition~3.1]{MartinezPerezNucinkis18}, which tells us that the group $nV$ has exactly two conjugacy classes of subgroups of order two, and hence of non-trivial involutions.
As $nV$ contains involutions both with and without fixed points, all non-trivial involutions with fixed points must belong to the same conjugacy class, which completes the proof.
\end{proof}

\begin{lemma}\label{lem:brin-absorption}
The set $\cT$ satisfies the absorption property.
\end{lemma}
\begin{proof}
Let $t\in\cT$.
From the definition of $\cT$ it follows that there exists a prefix transposition $\sigma=(p\ q) \in V$ and a non-empty proper cylinder $R=C(r)$ such that $t$ applies $\sigma$ in one of the first two coordinates when the entry in the other coordinate lies in $R$, and acts as the identity otherwise.
All coordinates except the one on which $\sigma$ acts are fixed.
Moreover, $H=\Fix(\sigma)=\frakC\setminus(C(p)\cup C(q))$ is non-empty.
We have to show that $t_{1,2}t_{1,3}$ is conjugate to $t_{1,3}$ by an element of $G(\{2,3\})$.
By \cref{lem:brin-generation}, this subgroup is the standard copy of $2V$ on coordinates $2,3$.

Suppose first that $t$ changes coordinate $1$, and let $\xi=(\xi_i)_{i\in\N}\in\frakC^{\N}$.
Then $t_{1,2}$ applies $\sigma$ to $\xi_1$ when $\xi_2\in R$, whereas $t_{1,3}$ does so when $\xi_3\in R$.
Both elements leave all other entries unchanged.
Since $\sigma^2=1$, we obtain
\[
(t_{1,2}t_{1,3})(\xi)=
\begin{cases}
(\sigma(\xi_1),\xi_2,\xi_3,\ldots),
& \text{if exactly one of $\xi_2,\xi_3$ lies in $R$},\\
\xi,
& \text{otherwise}.
\end{cases}
\]
We write $R^c=\frakC\setminus R$ and define
\[
A=(R\times R^c)\sqcup(R^c\times R),
\qquad
B=\frakC\times R.
\]
Then $t_{1,2}t_{1,3}$ applies $\sigma$ to $\xi_1$ precisely when $(\xi_2,\xi_3)\in A$, whereas $t_{1,3}$ does so precisely when $(\xi_2,\xi_3)\in B$.
The sets $A$ and $B$ are non-empty, clopen and have non-empty complements.
We can therefore apply the first assertion of \cref{lem:brin-matching} to the partitions $(A,A^c)$ and $(B,B^c)$ to obtain an element $g\in G(\{2,3\})$ whose action on the coordinates $2,3$ sends $A$ onto $B$.
Since $g$ fixes coordinate $1$, this gives
\[
g(t_{1,2}t_{1,3})g^{-1}=t_{1,3}.
\]
Suppose now that $t$ changes coordinate $2$.
When the entry in coordinate $1$ lies in $R$, the elements $t_{1,2}t_{1,3}$ and $t_{1,3}$ act on coordinates $2,3$ as $\sigma\times\sigma$ and $\mathrm{id}\times\sigma$, respectively.
Outside this subset, both elements act as the identity.
Thus, the elements $\sigma\times\sigma$ and $\mathrm{id}\times\sigma$ are non-trivial involutions in $2V$ with fixed sets $H\times H$ and $\frakC\times H$, respectively.
The second assertion of \cref{lem:brin-matching} provides us with an element $g\in G(\{2,3\})$ that conjugates the first involution to the second.
As $g$ fixes coordinate $1$, it preserves the condition that the entry in this coordinate lies in $R$.
It therefore follows that $g(t_{1,2}t_{1,3})g^{-1}=t_{1,3}$ on all of $\frakC^{\N}$, which completes the proof.
\end{proof}

\begin{proposition}\label{prop:brin-conjugacy}
The set $\cT$ satisfies the local conjugacy condition.
\end{proposition}
\begin{proof}
Let $t,u\in\cT$.
For $\xi=(\xi_i)_{i\in\N}\in\frakC^{\N}$, define
\[
s_{1,2}(\xi)=
\begin{cases}
t(\xi), & \xi_4\in C(0),\\
\xi, & \xi_4\in C(1),
\end{cases}
\qquad
w_{1,2}(\xi)=
\begin{cases}
\xi, & \xi_4\in C(0),\\
u(\xi), & \xi_4\in C(1).
\end{cases}
\]
Both elements fix all coordinates outside $\{1,2,4\}$ and are locally given by prefix replacements on these three coordinates.
They therefore belong to $G(\{1,2,4\})$ by \cref{lem:brin-generation}.
We first show that $s_{1,2}$ is conjugate to $t$ by an element of $G(\{2,4\})$.
From \cref{lem:brin-generation} we know that the latter subgroup is the standard copy of $2V$ on the coordinates $2,4$.
Let $\sigma=(p\ q)$ be the prefix transposition applied by $t$ and let $R=C(r)$ be the cylinder in the other coordinate determining when $\sigma$ is applied.
By construction, the sets $R$ and $\Fix(\sigma)$ are non-empty clopen proper subsets of $\mathfrak{C}$.

Suppose first that $t$ changes the first coordinate and define
\[
A=R\times\frakC,
\qquad
B=R\times C(0).
\]
Then $t$ applies $\sigma$ to $\xi_1$ precisely when $(\xi_2,\xi_4)\in A$, whereas $s_{1,2}$ does so when $(\xi_2,\xi_4)\in B$.
Both sets are non-empty, clopen and have non-empty complements.
This allows us to apply the first assertion of \cref{lem:brin-matching} to the partitions $(A,A^c)$ and $(B,B^c)$ in order to obtain an element $g\in G(\{2,4\})$ whose action on coordinates $2,4$ sends $A$ onto $B$.
Since $g$ fixes coordinate $1$, this gives
\[
gtg^{-1}=s_{1,2}.
\]
Suppose next that $t$ changes the second coordinate.
In this case we consider $\sigma\times\mathrm{id}$ on the coordinates $2,4$ and define the involution $\beta\in 2V$ by
\[
\beta(x,y)=
\begin{cases}
(\sigma(x),y), & y\in C(0),\\
(x,y), & y\in C(1).
\end{cases}
\]
These are non-trivial involutions with fixed sets
\[
\Fix(\sigma)\times\frakC
\quad\text{and}\quad
\bigl(\Fix(\sigma)\times C(0)\bigr)
\sqcup
\bigl(\frakC\times C(1)\bigr),
\]
respectively.
Since both fixed sets are non-empty, the second assertion of \cref{lem:brin-matching} endows us with an element $g\in G(\{2,4\})$ such that, on coordinates $2,4$,
\begin{equation}\label{eq:conjugation}
g(\sigma\times\mathrm{id})g^{-1}=\beta.
\end{equation}
We again claim that $gtg^{-1}=s_{1,2}$.
When $\xi_1\in R$, the elements $t$ and $s_{1,2}$ act on coordinates $2,4$ as $\sigma\times\mathrm{id}$ and $\beta$, respectively.
When $\xi_1\notin R$, both elements act as the identity.
Since $g$ fixes coordinate $1$, we may therefore deduce from~\eqref{eq:conjugation} that $gtg^{-1}=s_{1,2}$ on all of $\frakC^{\N}$.
Similarly, $w_{1,2}$ is conjugate to $u$ by an element of $G(\{2,4\})$, as follows by applying the above argument to $u$ with the roles of $C(0)$ and $C(1)$ exchanged.

Let us now consider the elements $s_{1,3}:=(2,3)\cdot s_{1,2}$ and $w_{1,3}:=(2,3)\cdot w_{1,2}$.
Since $(2,3)$ fixes $4$, the element $s_{1,3}$ agrees with $t_{1,3}$ when $\xi_4\in C(0)$ and acts as the identity otherwise.
Likewise, $w_{1,3}$ agrees with $u_{1,3}$ when $\xi_4\in C(1)$ and acts as the identity otherwise.
Using the absorption property established in \cref{lem:brin-absorption}, we obtain $h_t,h_u\in G(\{2,3\})$ with
\[
h_t(t_{1,2}t_{1,3})h_t^{-1}=t_{1,3},
\qquad
h_u(u_{1,2}u_{1,3})h_u^{-1}=u_{1,3}.
\]
We claim that these elements also satisfy
\[
h_t(s_{1,2}s_{1,3})h_t^{-1}=s_{1,3},
\qquad
h_u(w_{1,2}w_{1,3})h_u^{-1}=w_{1,3}.
\]
To see this, we note that $h_t$, $h_u$ fix coordinate $4$.
If $\xi_4\in C(0)$, the definitions of $s_{1,2},s_{1,3}$ and the choice of $h_t$ give us
\[
\begin{aligned}
h_t(s_{1,2}s_{1,3})h_t^{-1}(\xi)
&=h_t(t_{1,2}t_{1,3})h_t^{-1}(\xi)\\
&=t_{1,3}(\xi)\\
&=s_{1,3}(\xi).
\end{aligned}
\]
If $\xi_4\in C(1)$, the elements $s_{1,2}$ and $s_{1,3}$ fix $h_t^{-1}(\xi)$, so both sides of the first claimed identity fix $\xi$.
Similarly, if $\xi_4\in C(1)$, we obtain
\[
\begin{aligned}
h_u(w_{1,2}w_{1,3})h_u^{-1}(\xi)
&=h_u(u_{1,2}u_{1,3})h_u^{-1}(\xi)\\
&=u_{1,3}(\xi)\\
&=w_{1,3}(\xi).
\end{aligned}
\]
If $\xi_4\in C(0)$, the elements $w_{1,2}$ and $w_{1,3}$ fix $h_u^{-1}(\xi)$, so both sides of the second claimed identity fix $\xi$, which completes the verification of \cref{def:local-conjugacy}(i).

To verify \cref{def:local-conjugacy}(ii), we note that $s_{1,2}$ fixes every point $\xi$ with $\xi_4\in C(1)$, whereas $w_{1,3}$ fixes every point $\xi$ with $\xi_4\in C(0)$.
Since both elements fix coordinate $4$, we deduce that $s_{1,2}$ and $w_{1,3}$ commute.
\end{proof}

\begin{proof}[Proof of \cref{thm:main-intro}]
By \cref{lem:brin-templates,lem:brin-absorption,prop:brin-conjugacy}, all hypotheses of \cref{thm:main} are satisfied.
Since $G_n=nV$ by \cref{lem:brin-generation}, its conclusion gives the desired result.
\end{proof}

\subsection{Automorphisms of free groups}\label{sec:aut}
We now turn to automorphism groups of free groups $F_n$ and show how \cref{thm:main} yields property~$(T)$ for these groups in sufficiently large rank.
To do so, it will be convenient to work with the groups $\SAut(F_n)$, which denote the kernel of the 
composition
\[
\Aut(F_n) \longrightarrow \Aut(F_n/[F_n,F_n]) =  \operatorname{GL}_n(\mathbb{Z})
\xrightarrow{\det}\{\pm1\}.
\]
For the remainder of this subsection, we fix a basis $x_1,\ldots,x_n$ of $F_n$.
Using the latter, we obtain a canonical embedding of $\SAut(F_n)$ into $\SAut(F_{n+1})$ by extending each automorphism trivially to the additional generator $x_{n+1}$.
We can therefore define the direct union
\[
G=\bigcup_{n\geq1}\SAut(F_n),
\]
which one can think of as a subgroup of $\Aut(F_\infty)$, where $F_\infty$ is the free group with basis $\{x_i\mid i\in\N\}$.
The group $W=\left(\bigoplus_{i\in\N}C_2\right)\rtimes\Sym_{\mathrm{fin}}(\N)$ admits a natural action on $F_\infty$ by letting each permutation $\pi\in\Sym_{\mathrm{fin}}(\N)$ send $x_j$ to $x_{\pi(j)}$ and each sign change $\varepsilon_i$ send $x_i$ to $x_i^{-1}$ while fixing all other basis elements.
Since these automorphisms normalize $G$, we obtain a $W$-group structure on $G$ by conjugation.
Our generator templates for this $W$-group will be the Nielsen transformations that only involve $x_1$ and $x_2$.
Recall that, for distinct $i,j$, the right and left \emph{Nielsen transformations} $R_{ij},L_{ij}\in G$ fix all basis elements other than $x_i$, and act on $x_i$ by
\[
R_{ij}(x_i)=x_ix_j,
\qquad
L_{ij}(x_i)=x_jx_i.
\]
With this notation, our set of generator templates is
\[
\cT=\{R_{12}^{\pm1},L_{12}^{\pm1},R_{21}^{\pm1},L_{21}^{\pm1}\}.
\]
The choice of $\cT$ is motivated by the fact that $\SAut(F_n)$ is generated by its Nielsen transformations, see e.g.~\cite[Section~3]{Nitsche26}.

\begin{lemma}\label{lem:aut-templates}
The set $\cT$ is a set of generator templates for the $W$-group $G$.
Moreover, $G_n=\SAut(F_n)$ for every $n\geq1$.
\end{lemma}
\begin{proof}
The set $\cT$ is finite, non-empty and closed under taking inverses.
Since the transformations in $\cT$ involve only $x_1$ and $x_2$, every element of $W$ supported outside $\{1,2\}$ fixes $\cT$ pointwise.
The transposition $(1,2)$ interchanges the indices $1$ and $2$ in the Nielsen transformations and hence preserves $\cT$.
For the sign changes, direct calculation gives us
\begin{equation}\label{eq:applying-epilon-autfn}
\varepsilon_j\cdot R_{ij} =R_{ij}^{-1}, \quad
\varepsilon_j\cdot L_{ij} =L_{ij}^{-1}, \quad
\varepsilon_i\cdot R_{ij} =L_{ij}^{-1}, \quad
\varepsilon_i\cdot L_{ij} =R_{ij}^{-1}
\end{equation}
for $i \ne j$, which shows that $\varepsilon_1$ and $\varepsilon_2$ preserve $\cT$.
Together this tells us that $\cT$ is invariant under $W_2$.
The first two identities in~\eqref{eq:applying-epilon-autfn} also show that every $t\in\cT$ is sent to $t^{-1}$ by $\varepsilon_1$ or $\varepsilon_2$.
Next we note that for every $p=\{i,j\}$, we have
\[
\cT_p=\{R_{ij}^{\pm1},L_{ij}^{\pm1},R_{ji}^{\pm1},L_{ji}^{\pm1}\}.
\]
Since each Nielsen transformation uniquely determines the pair of basis elements it involves we obtain $\cT_p\cap\cT_q=\varnothing$ for $p\ne q$.
If $p\cap q=\varnothing$, the elements of $\cT_p$ and $\cT_q$ are supported on disjoint sets of generators and hence commute.
Finally, we note that for $n\geq2$, the union of the sets $\cT_p$ over all pairs $p\subseteq\{1,\ldots,n\}$ is given by the set consisting of all $R_{ij}^{\pm1},L_{ij}^{\pm1}$ with $1\leq i\ne j\leq n$.
By~\cite[Section~3]{Nitsche26}, these transformations generate $\SAut(F_n)$.
Thus we see that $G_n=\SAut(F_n)$ and that $G$ is generated by the union of all sets $\cT_p$, which completes the proof.
\end{proof}

\begin{lemma}\label{lem:aut-absorption}
The set $\cT$ satisfies the absorption property.
\end{lemma}
\begin{proof}
For each of the eight templates $t\in\cT$, we have to find an element $\gamma\in G(\{2,3\})$ that satisfies
\[
\gamma(t_{1,2}t_{1,3})\gamma^{-1}=t_{1,3}.
\]
Let us first reduce the general case to the two templates $R_{12}$ and $R_{21}$.
To this end, we observe from~\eqref{eq:applying-epilon-autfn} that every template in $\cT$ can be obtained from either $R_{12}$ or $R_{21}$ by applying a product of the sign changes $\varepsilon_1$ and $\varepsilon_2$.
To make use of the latter, we show that, if the required conjugator exists for $t$, then it also exists for $\varepsilon_1\cdot t$ and $\varepsilon_2\cdot t$.
To this end we assume that $\gamma\in G(\{2,3\})$ satisfies $\gamma(t_{1,2}t_{1,3})\gamma^{-1}=t_{1,3}$.
We first show that the same element $\gamma$ satisfies the required identity for $\varepsilon_1\cdot t$.
Indeed, by using that $\varepsilon_1$ fixes $\gamma$ we obtain
\[
\begin{aligned}
\gamma\bigl((\varepsilon_1\cdot t)_{1,2}
(\varepsilon_1\cdot t)_{1,3}\bigr)\gamma^{-1}
&=\gamma(\varepsilon_1\cdot t_{1,2})
(\varepsilon_1\cdot t_{1,3})\gamma^{-1}\\
&=(\varepsilon_1\cdot\gamma)
(\varepsilon_1\cdot t_{1,2})
(\varepsilon_1\cdot t_{1,3})
(\varepsilon_1\cdot\gamma)^{-1}\\
&=(\varepsilon_1\cdot\gamma)
\bigl(\varepsilon_1\cdot(t_{1,2}t_{1,3})\bigr)
(\varepsilon_1\cdot\gamma^{-1})\\
&=\varepsilon_1\cdot
\bigl(\gamma(t_{1,2}t_{1,3})\gamma^{-1}\bigr)\\
&=(\varepsilon_1\cdot t)_{1,3}.
\end{aligned}
\]
We next show that $\gamma'=(\varepsilon_2\varepsilon_3)\cdot\gamma$ satisfies the required identity for $\varepsilon_2\cdot t$.
From the construction we directly obtain $\varepsilon_1\cdot t_{1,j} = (\varepsilon_1\cdot t)_{1,j}$.
Since $\varepsilon_3$ fixes $t_{1,2}$ and $\varepsilon_2$ fixes $t_{1,3}$ we moreover have $(\varepsilon_2\varepsilon_3)\cdot t_{1,2}
=(\varepsilon_2\cdot t)_{1,2}$
and
\[
(\varepsilon_2\varepsilon_3)\cdot t_{1,3}
=(\varepsilon_3\varepsilon_2)\cdot t_{1,3}
=\varepsilon_3\cdot t_{1,3}
=\varepsilon_3\cdot\bigl((2,3)\cdot t\bigr)
=(2,3)\cdot(\varepsilon_2\cdot t)
=(\varepsilon_2\cdot t)_{1,3}.
\]
Using these identities and the definition of $\gamma'$, we obtain
\[
\begin{aligned}
\gamma'\bigl((\varepsilon_2\cdot t)_{1,2}
(\varepsilon_2\cdot t)_{1,3}\bigr)(\gamma')^{-1}
&=\gamma'
\bigl((\varepsilon_2\varepsilon_3)\cdot t_{1,2}\bigr)
\bigl((\varepsilon_2\varepsilon_3)\cdot t_{1,3}\bigr)
(\gamma')^{-1}\\
&=\bigl((\varepsilon_2\varepsilon_3)\cdot\gamma\bigr)
\bigl((\varepsilon_2\varepsilon_3)\cdot t_{1,2}\bigr)
\bigl((\varepsilon_2\varepsilon_3)\cdot t_{1,3}\bigr)
\bigl((\varepsilon_2\varepsilon_3)\cdot\gamma\bigr)^{-1}\\
&=(\varepsilon_2\varepsilon_3)\cdot
\bigl(\gamma(t_{1,2}t_{1,3})\gamma^{-1}\bigr)\\
&=(\varepsilon_2\cdot t)_{1,3}.
\end{aligned}
\]
It remains to construct the required conjugators for $R_{12}$ and $R_{21}$.
For $t=R_{12}$ we have $t_{1,3} = (2,3)R_{12} = R_{13}$.
By choosing $\gamma=L_{32}^{-1}$ a direct evaluation gives us
\[
\bigl(\gamma(t_{1,2}t_{1,3})\gamma^{-1}\bigr)(x_1) =x_1x_3 \quad \text{and} \quad \bigl(\gamma(t_{1,2}t_{1,3})\gamma^{-1}\bigr)(x_3)=x_3.
\]
Since all other basis elements are clearly fixed this shows that $\gamma(t_{1,2}t_{1,3})\gamma^{-1}$ coincides with $R_{13} = t_{1,3}$.
For $t=R_{21}$ we have $t_{1,3} = (2,3)R_{21} = R_{31}$.
By choosing $\gamma=R_{23}$ a direct evaluation gives us
\[
\bigl(\gamma(t_{1,2}t_{1,3})\gamma^{-1}\bigr)(x_2)
=x_2 \quad \text{and} \quad \bigl(\gamma(t_{1,2}t_{1,3})\gamma^{-1}\bigr)(x_3)
=x_3x_1.
\]
Again, all other basis elements are fixed, so that $\gamma(t_{1,2}t_{1,3})\gamma^{-1}$ coincides with $R_{31} = t_{1,3}$.
\end{proof}

\begin{lemma}\label{lem:aut-commutators}
The set $\cT$ satisfies the local commutator condition.
\end{lemma}
\begin{proof}
Fix $t,u\in\cT$ and suppose that $u\ne t^{\pm1}$.
We have to find $\epsilon,\delta\in\{\pm1\}$ with
\[
[t_{1,2}^{\epsilon},u_{1,3}^{\delta}]
\in G(\{2,3,4\}).
\]
To this end, we first note that $[R_{ij},L_{ik}]=1$ for $j,k\ne i$, which readily follows from the definition of Nielsen transformations.

Suppose that $t_{1,2}$ and $u_{1,3}$ both change $x_1$.
Since $u\ne t^{\pm1}$, one multiplies $x_1$ on the right and the other on the left.
By the preceding observation, these elements commute.

Suppose next that both $t_{1,2}$ and $u_{1,3}$ fix $x_1$.
Then $t_{1,2}\in\{R_{21}^{\pm1},L_{21}^{\pm1}\}$ and $u_{1,3}\in\{R_{31}^{\pm1},L_{31}^{\pm1}\}$.
In view of this, we see that the images of $x_2$ and $x_3$ under $t_{1,2}$ and $u_{1,3}$ satisfy
\[
\begin{aligned}
t_{1,2}(x_2)&\in
\{x_2x_1,x_2x_1^{-1},x_1x_2,x_1^{-1}x_2\},\\
u_{1,3}(x_3)&\in
\{x_3x_1,x_3x_1^{-1},x_1x_3,x_1^{-1}x_3\},
\end{aligned}
\]
while these transformations fix all other basis elements.
In particular we see that $u_{1,3}$ fixes $t_{1,2}(x_2)$ and that $t_{1,2}$ fixes $u_{1,3}(x_3)$.
Since moreover $u_{1,3}t_{1,2}$ and $t_{1,2}u_{1,3}$ fix $x_k$ for $k\notin\{2,3\}$ we deduce that
\[
(t_{1,2}u_{1,3})(x_k)=(u_{1,3}t_{1,2})(x_k)
\]
for all $k$ and hence $[t_{1,2},u_{1,3}]=1$.

It remains to consider the case where exactly one of $t_{1,2}$ and $u_{1,3}$ changes $x_1$.
To reduce this to the case where $t_{1,2}$ changes $x_1$ we observe for $\epsilon,\delta\in\{\pm1\}$ that
\[
(2,3)\cdot[t_{1,2}^{\epsilon},u_{1,3}^{\delta}]
=[t_{1,3}^{\epsilon},u_{1,2}^{\delta}]
=[u_{1,2}^{\delta},t_{1,3}^{\epsilon}]^{-1}.
\]
Since $(2,3)\cdot G(\{2,3,4\})=G(\{2,3,4\})$, it follows that $[t_{1,2}^{\epsilon},u_{1,3}^{\delta}]$ lies in $G(\{2,3,4\})$ if and only if so does $[u_{1,2}^{\delta},t_{1,3}^{\epsilon}]$.
Without loss of generality we can therefore assume that $t$ lies in $\{R_{12}^{\pm1},L_{12}^{\pm1}\}$ and $u$ lies in $\{R_{21}^{\pm1},L_{21}^{\pm1}\}$.

The sign changes $\varepsilon_2,\varepsilon_3$ preserve $G(\{2,3,4\})$, so applying either of them to a commutator preserves the required subgroup inclusion.
Since $\varepsilon_2$ inverts $t_{1,2}$ and fixes $u_{1,3}$ we can assume that $t \in \{R_{12},L_{12}\}$.
Similarly we can assume that $u_{1,3}\in\{R_{31},R_{31}^{-1}\}$ since $\varepsilon_3$ fixes $t_{1,2}$ and exchanges $R_{31}$ with $L_{31}^{-1}$, as well as their inverses.
If $t=R_{12}$, we choose $\delta$ so that $u_{1,3}^{\delta}=R_{31}^{-1}$.
Then a direct calculation shows
\[
[R_{12},R_{31}^{-1}]=R_{32}^{-1} \in G(\{2,3,4\}).
\]
Similarly, if $t=L_{12}$ we choose $\delta$ so that $u_{1,3}^{\delta}=R_{31}$ and obtain
\[
[L_{12},R_{31}]=R_{32} \in G(\{2,3,4\}),
\]
which completes the proof.
\end{proof}

Recall that property~$(T)$ for $\Aut(F_n)$ in every rank $n\geq4$ was established through the work of Kaluba, Nowak, and Ozawa \cite{KalubaNowakOzawa19}, Kaluba, Kielak, and Nowak \cite{KalubaKielakNowak21}, and Nitsche \cite{Nitsche22}.
More recently, Nitsche \cite{Nitsche26} gave a computer-free proof for all sufficiently large $n$.
The following corollary recovers this asymptotic result from our general criterion.

\begin{corollary}\label{cor:aut}
The groups $\Aut(F_n)$ have property~$(T)$ for all sufficiently large $n$.
\end{corollary}
\begin{proof}
By \cref{lem:aut-templates,lem:aut-absorption,lem:aut-commutators}, all hypotheses of \cref{thm:main} are satisfied.
Since $G_n=\SAut(F_n)$, the theorem implies that $\SAut(F_n)$ has property~$(T)$ for all sufficiently large $n$.
The conclusion for $\Aut(F_n)$ follows because $\SAut(F_n)$ has index two in $\Aut(F_n)$ and property~$(T)$ passes to finite-index overgroups.
\end{proof}

\subsection{Elementary groups}\label{sec:elementary}
We next apply \cref{thm:main} to elementary groups over finitely generated rings.
To this end, we fix a non-trivial finitely generated associative unital ring $R$.
Given $n\geq2$ and $1 \leq i, j \leq n$, we write $E_{ij} \in R^{n \times n}$ to denote the $n\times n$-matrix whose entry in position $(i,j)$ is $1$ and all other entries are zero.
For $i\ne j$ and $r\in R$, the corresponding elementary matrix is given by $e_{ij}(r)=I_n+rE_{ij}$.
The elementary group $\EL_n(R)$ is defined as the subgroup
\[
\EL_n(R)
=
\langle e_{ij}(r)\mid 1\leq i\ne j\leq n,\ r\in R\rangle
\]
of $\GL_n(R)$ that is generated by all elementary matrices.
By embedding $\EL_n(R)$ in the upper left corner of $\EL_{n+1}(R)$ we can define the direct union
\[
G=\bigcup_{n\geq2}\EL_n(R),
\]
which we think of as a group acting on $R^{(\N)}=\bigoplus_{i\in\N}R$ by matrix multiplication in the canonical way.
Let $(b_i)_{i\in\N}$ denote the standard basis of $R^{(\N)}$.
The group $W=\left(\bigoplus_{i\in\N}C_2\right)\rtimes\Sym_{\mathrm{fin}}(\N)$ acts on $R^{(\N)}$ by
\[
\pi(b_i)=b_{\pi(i)},
\qquad
\varepsilon_k(b_i)=
\begin{cases}
-b_i, & i=k,\\
b_i, & i\ne k,
\end{cases}
\]
for $\pi\in\Sym_{\mathrm{fin}}(\N)$ and $k\in\N$.
By conjugating elementary matrices with these automorphisms we obtain
\begin{equation}\label{eq:el-conjugate}
\pi e_{ij}(r)\pi^{-1} = e_{\pi(i)\pi(j)}(r)
\end{equation}
and
\[
\varepsilon_k e_{ij}(r)\varepsilon_k^{-1}
=
\begin{cases}
e_{ij}(-r), & k\in\{i,j\},\\
e_{ij}(r), & k\notin\{i,j\}.
\end{cases}
\]
In particular we see that $G$ is normalized by $W$.
Thus conjugation by $W$ endows $G$ with the structure of a $W$-group.
To specify our generator templates for this $W$-group, we choose a finite generating set $\Sigma\subset R\setminus\{0\}$ of $R$ that contains $1$ and is closed under taking additive inverses.
Our set of generator templates for $G$ will be
\[
\cT=\{e_{12}(r),e_{21}(r)\mid r\in\Sigma\}.
\]
From~\eqref{eq:el-conjugate} we deduce that for every two-element set $p=\{i,j\}\subset\N$ we have
\begin{equation}\label{eq:tp-for-el}
\cT_p=\{e_{ij}(r),e_{ji}(r)\mid r\in\Sigma\}.
\end{equation}
Our choice of $\cT$ is motivated by the following lemma.

\begin{lemma}\label{lem:elcompiler}
For every finite set $I\subset\N$ with $|I|\geq3$, the group $G(I)$
is the copy of $\EL_{|I|}(R)$ embedded in the rows and columns indexed
by $I$. In particular, $G_n=\EL_n(R)$ for $n\geq3$.
\end{lemma}
\begin{proof}
Since $\Sigma$ generates $R$ as a ring and contains $1$, the elementary matrices $e_{ij}(r)$ with $r\in\Sigma$ and $1\leq i\ne j\leq n$ generate $\EL_n(R)$ for every $n\geq3$, see \cite[Section~6.1, equation~(6.1)]{EJ}.
By combining this with~\eqref{eq:tp-for-el} we see that
\[
G_n
=\langle e_{ij}(r)\mid 1\leq i\ne j\leq n,\ r\in\Sigma\rangle
=\EL_n(R).
\]
Now let $I\subset\N$ be finite with $|I|=n\geq3$ and choose $\pi\in\Sym_{\mathrm{fin}}(\N)$ with $\pi(\{1,\ldots,n\})=I$.
Then it follows from~\eqref{eq:el-conjugate} that
the action of $\pi$ sends $G_n$ onto $G(I)$, which is the standard copy of $\EL_n(R)$ in the rows and columns in $I$.
\end{proof}

\begin{lemma}\label{lem:el-templates}
The set $\cT$ is a set of generator templates for the $W$-group $G$.
\end{lemma}

\begin{proof}
Since $\Sigma$ is closed under taking additive inverses we deduce from $e_{ij}(r)^{-1}=e_{ij}(-r)$ that $\cT$ is closed under taking inverses.
The transposition $(1,2)$ swaps $e_{12}(r)$ and $e_{21}(r)$, while the sign changes $\varepsilon_1,\varepsilon_2$ send every element of $\cT$ to its inverse.
Thus $\cT$ is invariant under $W_2$.
From~\eqref{eq:el-conjugate} we deduce that every element of $W$ supported outside $\{1,2\}$ fixes $\cT$ pointwise.
Since $0\notin\Sigma$, each element of $\cT_p$ has exactly one non-zero off-diagonal entry, whose row and column indices determine $p$.
Thus $\cT_p\cap\cT_q=\varnothing$ for $p\ne q$.
For disjoint pairs $\{i,j\}$ and $\{k,\ell\}$, we clearly have $[e_{ij}(r),e_{k\ell}(s)]=1$ so that the elements of $\cT_p$ commute with those of $\cT_q$ whenever $p\cap q=\varnothing$.
Finally, \cref{lem:elcompiler} gives us
\[
\left\langle\bigcup_{\substack{p\subset\N\\|p|=2}}\cT_p\right\rangle
=\bigcup_{n\geq3}G_n
=\bigcup_{n\geq3}\EL_n(R)
=G,
\]
which completes the proof.
\end{proof}

\begin{lemma}\label{lem:el-absorption}
The set $\cT$ satisfies the absorption property.
\end{lemma}
\begin{proof}
For every $t\in\cT$, we must find an element $\gamma\in G(\{2,3\})$ such that
\[
\gamma(t_{1,2}t_{1,3})\gamma^{-1}=t_{1,3}.
\]
For $r\in\Sigma$, matrix multiplication gives us
\begin{align*}
e_{32}(1)\bigl(e_{12}(r)e_{13}(r)\bigr)e_{32}(-1)
&=e_{13}(r),\\
e_{23}(-1)\bigl(e_{21}(r)e_{31}(r)\bigr)e_{23}(1)
&=e_{31}(r).
\end{align*}
Since $1,-1\in\Sigma$, the elements $e_{32}(1)$ and $e_{23}(-1)$ lie in $G(\{2,3\})$.
Thus we have established the absorption property for $e_{12}(r)$ and $e_{21}(r)$, and hence for $\cT$.
\end{proof}

\begin{lemma}\label{lem:el-commutators}
The set $\cT$ satisfies the local commutator condition.
\end{lemma}
\begin{proof}
We verify the local commutator condition by showing that
\begin{equation}\label{eq:el-loc-com}
[t_{1,2},u_{1,3}]\in G(\{2,3,4\})
\end{equation}
for all $t,u\in\cT$.
The verification of \eqref{eq:el-loc-com} reduces to the following four cases, where $r,s\in\Sigma$ are arbitrary and the identities follow by direct matrix multiplication.
\begin{align*}
[e_{12}(r),e_{13}(s)]&=1,\\
[e_{21}(r),e_{31}(s)]&=1,\\
[e_{12}(r),e_{31}(s)]&=e_{32}(-sr)
=[e_{34}(-s),e_{42}(r)],\\
[e_{21}(r),e_{13}(s)]&=e_{23}(rs)
=[e_{24}(r),e_{43}(s)].
\end{align*}
Every elementary matrix in the rightmost commutators has both indices in $\{2,3,4\}$.
In particular this tells us that the commutators belong to $G(\{2,3,4\})$, proving the local commutator condition.
\end{proof}

Recall that Ershov and Jaikin-Zapirain \cite{EJ} proved that $\EL_n(R)$ has property~$(T)$ for every finitely generated associative unital ring $R$ and every $n\geq3$.
The following corollary recovers their result for all sufficiently large $n$.

\begin{corollary}\label{cor:elementary}
For every finitely generated associative ring $R$ with identity,
$\EL_n(R)$ has property~$(T)$ for all sufficiently large $n$.
\end{corollary}
\begin{proof}
We apply our Kazhdan criterion \cref{thm:main}.
From \cref{lem:el-templates} we know that $\cT$ is a set of generator templates for the $W$-group $G$.
By \cref{lem:el-templates,lem:el-absorption,lem:el-commutators}, all hypotheses of \cref{thm:main} are satisfied, from which we conclude that $G_n = \EL_n(R)$ has property~$(T)$ for all sufficiently large $n$.
\end{proof}

\appendix

\section{Inapproximability of higher Brin--Thompson groups \\ (by Francesco Fournier-Facio)}

The past two months have seen the resolution of the three main problems in the subject of metric approximations of groups: the existence of non-sofic \cite{OpenAI26, fff, KunThom26}, non-MF \cite{caleb, scott} and non-hyperlinear \cite{Thom26, ALT26} groups.

In all three problems, the main ingredient is a property $(T)$ group with certain self-similarity properties. The world of Thompson groups is rich in self-similarity, and indeed the first construction of non-sofic groups \cite{OpenAI26} was, in practice, the only available source of Thompson-like groups with property $(T)$: the elementary groups of the binary Leavitt algebra (thanks to a connection established independently by Birget \cite{birget} and Nekrashevych \cite{nekr}, and the result of Ershov--Jaikin-Zapirain \cite{EJ}). \cref{thm:main-intro} gives a more natural source of Thompson-like property $(T)$ groups, and leads to the following.

\begin{theorem}
\label{main:inapprox}
For sufficiently large $n$, the Brin--Thompson group $nV$ is not hyperlinear (in particular not sofic) nor MF.
\end{theorem}

This makes Brin--Thompson groups the first groups that were extensively studied in the literature and that happen to be inapproximable in either of these three senses: the previous examples are constructed ad-hoc. They are also the first groups of type $F_\infty$ \cite{FMWZ13}, and the first finitely presented simple groups, that are inapproximable. Note that finitely presented simple groups have long been proposed as candidates for non-sofic groups, in particular Thompson's group $V$, and Burger--Mozes groups \cite{BM}.

The proof for soficity and hyperlinearity is almost immediate, and it will work whenever $nV$ has property $(T)$ (\cref{cor:sofichyper}). For the non-MF case, where the existing criteria are less flexible, we will prove that whenever $nV$ has property $(T)$, $(n+1)V$ is not MF (\cref{thm:MF}).

\medskip

Other approximation properties received attention in the past few years in particular approximability with respect to the unnormalised Schatten $p$-norm, for $1 \leq p < \infty$. For each of these properties, a non-approximable group had been found before the recent breakthroughs, and with independent techniques \cite{2:approx, p:approx, 1:approx}. Thanks to \cite{1:approx}, \cref{main:inapprox} readily implies:

\begin{corollary}
\label{schatten}
    For sufficiently large $n$ (independent of $p$), the Brin--Thompson group $nV$ is not Schatten $p$-approximable, for any $1 \leq p < \infty$.
\end{corollary}

\subsection{Soficity and hyperlinearity}

The original discovery of a non-sofic group \cite{OpenAI26} relies on a pair of property $(T)$ groups that have the following property. The definition is from \cite{KunThom26}.

\begin{definition}
We say that a subgroup $\Gamma < G$ is \emph{infranormal} if the set $\{ t \in G : t \Gamma t^{-1} < \Gamma\}$ generates $G$.
\end{definition}

This was the key notion needed to connect Kun's theorem \cite{Kun16} that sofic approximations of property $(T)$ groups are essentially \emph{disjoint unions of} expanders, to the Kun--Thom centraliser rigidity theorem \cite{KunThom19}, which needs certain sofic approximations to be essentially expanders. Using a deeper centraliser rigidity theorem of Alekseev--Thom \cite{alekseevthom}, Kun--Thom were able to give a more robust criterion for non-soficity \cite[Theorem B]{KunThom26}. In \cite{Thom26}, Thom proved the analogous result for hyperlinearity, conditionally on an analogue of the centraliser rigidity theorem of Alekseev--Thom \cite[Open Problem 6.2]{alekseevthom}. This has since been resolved \cite{ALT26}, hence we obtain an analogous criterion for non-hyperlinearity (which subsumes the one for non-soficity).

\begin{theorem}[{\cite[Corollary D]{ALT26}}]
\label{criterion:hyper}
Let $\Gamma < G$ be an infranormal subgroup, and suppose that both have property $(T)$. If $G$ is hyperlinear (in particular, if $G$ is sofic) then $C_G(\Gamma)$ is normal in $G$.
\end{theorem}

We can now apply this directly to $nV$.

\begin{corollary}
\label{cor:sofichyper}
If $nV$ has property $(T)$, then it is not hyperlinear (in particular not sofic).
\end{corollary}

Combined with \cref{thm:main-intro}, we deduce the sofic and hyperlinear parts of \cref{main:inapprox}.
We start by setting up some notation since it will be used again later on.

\begin{notation}
\label{not:Uw}

Let $G \coloneqq nV$. For a finite binary string $w \in \{0, 1\}^*$ let $U_w \coloneqq w\frakC \times \frakC^{n-1} \subset \frakC^n$, so that $\frakC^n = U_0 \sqcup U_1$. Let $\Gamma \coloneqq \Fix(U_1)$, that is the copy of $nV$ localised at $U_0$. Let $t$ be the inverse of the first standard generator of Thompson's group $F$ extended to $G$, namely:
\[t U_0 = U_{00}; \quad t U_{10} = U_{01}; \quad t U_{11} = U_1.\]
\begin{center}
\begin{tikzpicture}[
    line width=1pt,
    >=Stealth,
    every node/.style={font=\normalsize}
]

\def\ytop{1.5}
\def\ybot{-1.5}
\def\unit{3}   

\pgfmathsetmacro{\txa}{0*\unit}
\pgfmathsetmacro{\txb}{1*\unit}
\pgfmathsetmacro{\txc}{1.5*\unit}
\pgfmathsetmacro{\txd}{2*\unit}

\pgfmathsetmacro{\bxa}{0*\unit}
\pgfmathsetmacro{\bxb}{0.5*\unit}
\pgfmathsetmacro{\bxc}{1*\unit}
\pgfmathsetmacro{\bxd}{2*\unit}

\draw (\txa,\ytop) -- (\txb,\ytop);
\draw (\txa,\ytop-0.12) -- (\txa,\ytop+0.12);
\draw (\txb,\ytop-0.12) -- (\txb,\ytop+0.12);
\node[above=4pt] at ({(\txa+\txb)/2},\ytop+0.12) {$U_0$};

\draw (\txb,\ytop) -- (\txc,\ytop);
\draw (\txc,\ytop-0.12) -- (\txc,\ytop+0.12);
\node[above=4pt] at ({(\txb+\txc)/2},\ytop+0.12) {$U_{10}$};

\draw (\txc,\ytop) -- (\txd,\ytop);
\draw (\txd,\ytop-0.12) -- (\txd,\ytop+0.12);
\node[above=4pt] at ({(\txc+\txd)/2},\ytop+0.12) {$U_{11}$};

\draw (\bxa,\ybot) -- (\bxb,\ybot);
\draw (\bxa,\ybot-0.12) -- (\bxa,\ybot+0.12);
\draw (\bxb,\ybot-0.12) -- (\bxb,\ybot+0.12);
\node[below=4pt] at ({(\bxa+\bxb)/2},\ybot-0.12) {$U_{00}$};

\draw (\bxb,\ybot) -- (\bxc,\ybot);
\draw (\bxc,\ybot-0.12) -- (\bxc,\ybot+0.12);
\node[below=4pt] at ({(\bxb+\bxc)/2},\ybot-0.12) {$U_{01}$};

\draw (\bxc,\ybot) -- (\bxd,\ybot);
\draw (\bxd,\ybot-0.12) -- (\bxd,\ybot+0.12);
\node[below=4pt] at ({(\bxc+\bxd)/2},\ybot-0.12) {$U_1$};

\draw[->,thick] (-1.2,\ytop-0.3) -- (-1.2,\ybot+0.3)
    node[midway,left=2pt] {$t$};

\end{tikzpicture}
\end{center}
\end{notation}

\begin{proof}
Let $U_w, G, \Gamma$ and $t$ be as in \cref{not:Uw}. By assumption both $\Gamma$ and $G$ have property $(T)$. Since $\Gamma$ has trivial centre and acts minimally on $U_0$, the centraliser of $\Gamma$ must act trivially on $U_0$, therefore $C_G(\Gamma) = \Fix(U_0)$, which is not normal in $G$. Thanks to \cref{criterion:hyper}, it remains to show that $\Gamma < G$ is infranormal.

Let $t_1 \coloneqq t$, and let $t_0$ be $t$ precomposed with an involution flipping $U_{10}$ and $U_{11}$, so that $t_i U_{1i} = U_1$. Then $t_i U_1 \supset U_1$ and so $t_i \Gamma t_i^{-1} = \Fix(t_i U_1) < \Fix(U_1) = \Gamma$. We need show that $G' \coloneqq \langle \Gamma, t_0, t_1 \rangle$ is equal to the whole of $G$.
First, $G'$ contains $t_i^{-1} \Gamma t_i = \Fix(t_i^{-1}U_1) = \Fix(U_{1i})$. For any nonempty clopen $U' \subset U_{10}$, we can pick an element of $\Fix(U_{11})$ mapping $U'$ to $U_{10}$, and conjugating $\Fix(U_{10})$ by its inverse we obtain $\Fix(U') < G'$.

Now let $g \in G$ be an arbitrary element. Let $U' \subset U_{10}$ be a sufficiently small clopen set with the property that $g U'$ does not intersect both $U_{10}$ and $U_{11}$. Then there exists an element $h \in G'$ such that $h g U' = U'$. Indeed: if $g U'$ is disjoint from $U_{11}$, we can pick $h \in \Fix(U_{11})$; else if $g U' \subset U_{11}$ (the case illustrated below), we first map it into $U_0$ with an element of $\Fix(U_{10})$ and then to $U'$ with an element of $\Fix(U_{11})$.

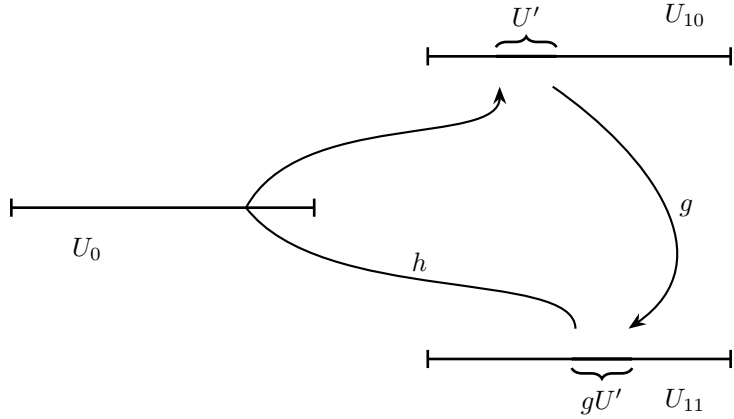
\begin{center}
\begin{tikzpicture}[
    line width=1pt,
    >=Stealth,
    every node/.style={font=\normalsize}
]

\draw (0,0) -- (4,0);
\draw (0,-0.12) -- (0,0.12);
\draw (4,-0.12) -- (4,0.12);
\node[below=4pt] at (1,-0.12) {$U_0$};

\def\xa{5.5}     
\def\xb{9.5}     
\def\ytop{2}     
\def\ybot{-2}    

\draw (\xa,\ytop) -- (\xb,\ytop);
\draw (\xa,\ytop-0.12) -- (\xa,\ytop+0.12);
\draw (\xb,\ytop-0.12) -- (\xb,\ytop+0.12);
\node[above=4pt] at (\xb-0.6,\ytop+0.12) {$U_{10}$};

\draw (\xa,\ybot) -- (\xb,\ybot);
\draw (\xa,\ybot-0.12) -- (\xa,\ybot+0.12);
\draw (\xb,\ybot-0.12) -- (\xb,\ybot+0.12);
\node[below=4pt] at (\xb-0.6,\ybot-0.12) {$U_{11}$};

\def\va{6.4}
\def\vb{7.2}
\draw[very thick] (\va,\ytop) -- (\vb,\ytop);
\draw[decorate,decoration={brace,amplitude=4pt}]
    (\va,\ytop+0.12) -- (\vb,\ytop+0.12)
    node[midway,above=4pt] {$U'$};

\def\ga{7.4}
\def\gb{8.2}
\draw[very thick] (\ga,\ybot) -- (\gb,\ybot);
\draw[decorate,decoration={brace,amplitude=4pt,mirror}]
    (\ga,\ybot-0.12) -- (\gb,\ybot-0.12)
    node[midway,below=4pt] {$gU'$};

\pgfmathsetmacro{\Vmid}{(\va+\vb)/2}
\pgfmathsetmacro{\gVmidval}{(\ga+\gb)/2}

\pgfmathsetmacro{\gVmid}{\Vmid+0.35}
\pgfmathsetmacro{\ggVmid}{\gVmidval+0.35}
\pgfmathsetmacro{\hgVmid}{\gVmidval-0.35}
\pgfmathsetmacro{\hVmid}{\Vmid-0.35}

\draw[->,thick]
    (\gVmid,\ytop-0.4)
    .. controls (\gVmid+1.3,\ytop-1.3) and (\ggVmid+1.3,\ybot+1.3) ..
    (\ggVmid,\ybot+0.4)
    node[pos=0.5,right=2pt] {$g$};

\draw[->,thick]
    (\hgVmid,\ybot+0.4)
    .. controls (\hgVmid,\ybot+1.2) and (4.1,-1.3) ..
    (3.1,0)
    node[pos=0.5,above left=1pt] {$h$}
    .. controls (3.8,1.3) and (\hVmid,\ytop-1.2) ..
    (\hVmid,\ytop-0.4);

\end{tikzpicture}
\end{center}

Finally, since elements of $G$ are defined locally, there exists an element $k \in \Fix(U_{11}) \subset G'$ whose local action on $U'$ is equal to that of $h g$, and so $k^{-1} h g \in \Fix(U') < G'$. Hence $g = h^{-1} \cdot k \cdot k^{-1} h g \in G'$ and we conclude.
\end{proof}

\subsection{A criterion for non-MF groups}

Following \cite{MF} we say that a group is MF\footnote{Some people say this stands for \emph{matricially finite}, others for \emph{matricial field}.} if it admits a faithful homomorphism into $U(\Qd)$, where $d = (d_n)_{n \geq 1}$ is a sequence of integers and $\Qd$ is the $\Cstar$-algebra
\[\Qd \coloneqq \prod_{n \geq 1} M_{d_n} / \bigoplus_{n \geq 1} M_{d_n}.\]
The story for non-MF groups is slightly different. The following criterion was discovered by Eckhardt (similar criteria were discovered concurrently and independently by Sauers \cite{scott}):

\begin{theorem}[{\cite[Theorem 4.3]{caleb}}]
\label{criterion:MF}

Let $\Gamma$ be a group with property $(T)$. Let $\Gamma < G$ and suppose that there exists $t \in G$ be such that $t \Gamma t^{-1} \lneq \Gamma$. Then $\Z/2 \wr_{G/\Gamma} G$ is not MF.
\end{theorem}

This is inspired from the non-sofic wreath products from \cite{KunThom26}, but it is proved directly, without going through a robust criterion such as \cref{criterion:hyper}. We will need a slightly more permissive version of the same statement. Recall that a \emph{Boolean group} is a group in which every element has order $2$, in other words, it is the underlying group of an $\mathbb{F}_2$-vector space.

\begin{theorem}
\label{criterion:MF:new}

Let $\Gamma$ be a group with property $(T)$. Let $\Gamma < G$ and let $t \in G$ be such that $H \coloneqq t \Gamma t^{-1} \lneq \Gamma$. Let A be a countable Boolean group, and let $W \coloneqq G \ltimes A$ be a semidirect product, such that the action of $G$ on $A$ by automorphism satisfies $A^{\Gamma} \neq A^H$. Then $W$ is not MF.
\end{theorem}

This implies \cref{criterion:MF}: the group $G$ is acting on $A = \oplus_{G/\Gamma} \Z/2$ and the basis element concentrated at $t \Gamma \in G/\Gamma$ is fixed by $H$ and not by $\Gamma$. The proof is almost identical, but we go through it for completeness. We purposefully follow the same notation and numbering of equations as in the proof of \cite[Theorem 4.3]{caleb}, whenever possible. A key insight of the proof is in the following proposition, which is a direct consequence of \cite[Theorem 3.1]{caleb}.

\begin{proposition}[{\cite[Theorem 3.1]{caleb}}]
\label{fixedpoints:finite}

Let $\Gamma, t, H$ be as in \cref{criterion:MF:new}. Let $\pi \colon \langle \Gamma, t \rangle \to U(\mathcal{H})$ be a unitary representation that factors through a \emph{finite} unital $\Cstar$-algebra. Then $\mathcal{H}^{\pi(\Gamma)} = \mathcal{H}^{\pi(H)}$.
\end{proposition}

Throughout the proof, we use two different norms on $M_{d_n}$: the operator norm $\| \cdot \|$, and the \textbf{un}normalised Hilbert--Schmidt norm $\| \cdot \|_{HS}$.

\begin{proof}[Proof of \cref{criterion:MF:new}]
Let $a_0 \in A$ be an element that is fixed by $H$, and not by $\gamma_0 \in \Gamma$. The claim is that any homomorphism $\pi \colon W \to U(\Qd)$ is unfaithful, as witnessed by the fact that $\pi(a_0) = \pi(\gamma_0 \cdot a_0)$.  So let us suppose by contradiction that there exists $\pi$ such that $\pi(a_0) \neq \pi(\gamma_0 \cdot a_0)$. We also let $\pi$ denote the linearisation $\Cstar(W) \to \Qd$. The goal is to use $\pi$ to build a representation of $W$ factoring through a unital finite $\Cstar$-algebra, such that the sets of fixed vectors of $H$ and $\Gamma$ differ, which contradicts \cref{fixedpoints:finite}.

\medskip

\emph{Preparation.} By definition of $\Qd$, from $\pi$ we may extract an asymptotic representation, i.e. a sequence of maps $\pi_n \colon W \to U(M_{d_n})$ such that
\begin{equation}
\label{eq:asym:mult}
\| \pi_n(gh) - \pi_n(g)\pi_n(h) \| \to 0 \text{ for all } g, h \in W
\end{equation}
and hence
\begin{equation}
\label{eq:asym:conj}
\| \pi_n(g)\pi_n(a)\pi_n(g)^* - \pi_n(g \cdot a) \| \to 0 \text{ for all } g \in G, a \in A.
\end{equation}

Since $A$ is countable and Boolean, by Pontryagin duality $\Cstar(A)$ is isomorphic to the $\Cstar$-algebra of continuous function on a Cantor space, which is weakly semiprojective \cite{loring}. It follows that there exists a homomorphic lifting of $\pi|_{\Cstar(A)}$, and so the maps $\pi_n$ can be modified so that \eqref{eq:asym:mult} still holds, and moreover $\pi_n|_A$ is a homomorphism; we assume this from now on. In particular the unitaries $\pi_n(a)$ are self-adjoint. Now \eqref{eq:asym:conj} implies that, for all $g \in G, a, b \in A$, and for sufficiently large $n$ (depending on $g, a$ and $b$):
\[
\| \pi_n(g)(\pi_n(a) - \pi_n(b))\pi_n(g)^* - (\pi_n(g \cdot a) - \pi_n(g \cdot b)) \| < 1 \text{ for all } g \in G, a, b \in A.
\]
Note that a term of the form $\pi_n(a) - \pi_n(b)$ is selfadjoint and has spectrum contained in $\{-2, 0, 2\}$ (being a difference of two commuting selfadjoint unitaries). Since their difference is less than $1$, their eigenspaces must have the same dimensions, and so they are unitarily equivalent. Therefore:
\begin{equation}
\label{eq:diff:equal}
\| \pi_n(a) - \pi_n(b) \|_{HS} = \| \pi_n(g \cdot a) - \pi_n(g \cdot b)\|_{HS}.
\end{equation}

Now if $\pi_n(a) \neq \pi_n(b)$, we have
\begin{equation}
\label{eq:diff:big}
\| \pi_n(a) - \pi_n(b) \|_{HS} = \| \pi_n(a) \pi_n(b) - 1 \|_{HS} \geq 2;
\end{equation}
because $\pi_n(a)\pi_n(b)$ is a non-identity selfadjoint unitary. By our assumption there are infinitely many values of $n$ such that $\pi_n(\gamma_0 \cdot a_0) \neq \pi_n(a_0)$, so by \eqref{eq:diff:big}, up to subsequence we may assume without loss of generality that
\begin{equation}
\label{eq:diff:big0}
\| \pi_n(\gamma_0 \cdot a_0) - \pi_n(a_0) \|_{HS} \geq 2 \text{ for all } n.
\end{equation}
Let us also record the following triangle inequality, which follows from \eqref{eq:diff:equal}. For all $g, h \in G, a \in A$ and for sufficiently large $n$:
\begin{equation}
\label{eq:triangle}
\| \pi_n(gh \cdot a) - \pi_n(a) \|_{HS} \leq \| \pi_n(g \cdot a) - \pi_n(a)\|_{HS} + \| \pi_n(h \cdot a) - \pi_n(a)\|_{HS}.
\end{equation}

\medskip

\emph{The representation.} Since $\Gamma$ has property $(T)$, it is finitely generated, so there exist $\gamma_1, \ldots, \gamma_m \in \Gamma \setminus H$ that generate $\Gamma$ together with $H$. For each $n$, set
\[k_n \coloneqq \max\limits_{0 \leq i \leq m} \| \pi_n(\gamma_i \cdot a_0) - \pi_n(a_0) \|_{HS}\]
and note that $k_n \geq 2$ by \eqref{eq:diff:big0}. By \eqref{eq:diff:equal}, up to subsequence we may assume that for all $n$:
\[\| \pi_n(\gamma_i \cdot a_0) - \pi_n(a_0) \|_{HS} = \| \pi_n(\gamma_i^{-1} \cdot a_0) - \pi_n(a_0) \|_{HS}.\]

Pick an arbitrary $g \in \Gamma$, and write it as a product in terms of the generating set $\langle H, \gamma_0, \ldots, \gamma_m \rangle$ as follows:
\begin{equation}
\label{eq:expression}
g = h_0 \gamma_{i_1}^{\varepsilon_1} h_1 \cdots \gamma_{i_r}^{\varepsilon_r} h_r.
\end{equation}
Using \eqref{eq:triangle}, the fact that $H$ fixes $a_0$, the definition of $k_n$ and the formula for inverses above, we obtain
\begin{equation}
\label{eq:bound:r}
\| \pi_n(g \cdot a_0) - \pi_n(a_0) \|_{HS} \leq  \sum\limits_{j = 1}^r \| \pi_n(\gamma_{i_j}^{\varepsilon_j} \cdot a_0) - \pi_n(a_0) \|_{HS} \leq r k_n.
\end{equation}

Now define $\calK_n$ to be $M_{d_n}$ with the rescaled Hilbert--Schmidt norm $\|T\|_{\calK_n} = \| T \|_{HS}/k_n$. For each $n$, let $\sigma_n(g) \coloneqq \Ad \pi_n(g) \in U(\calK_n)$. Then, for all $g, h \in G$ and all $T \in \calK_n$:
\begin{align*}
    \|\sigma(gh)T - \sigma(g)\sigma(h)T\|_{\calK_n} &= \| \pi(gh) T \pi(gh)^* - \pi(g)\pi(h) T (\pi(g) \pi(h))^*\|_{\calK_n} \\
    &\leq \| (\pi(gh) - \pi(g)\pi(h)) T \|_{\calK_n} + \| T (\pi(gh)^* - (\pi(g) \pi(h))^*) \|_{\calK_n} \\
    &\leq 2 \| \pi(gh) - \pi(g)\pi(h)\| \cdot \|T\|_{\calK_n} \xrightarrow{\eqref{eq:asym:mult}} 0.
\end{align*}
That is, $(\sigma_n)_n$ is an asymptotic representation, and so
\[\sigma \coloneqq (\sigma_n)_n \colon G \to \prod B(\calK_n) / \bigoplus B(\calK_n)\]
is a genuine representation. Crucially, the target is a finite $\Cstar$-algebra.

Pushing this one step further, we pick a non-principal ultrafilter $\omega$ and let $\calK \coloneqq \prod_\omega \calK_n$ denote the Hilbert space ultraproduct, and let
\[\rho \colon \prod B(\calK_n) / \bigoplus B(\calK_n) \to B(\calK)\]
be the natural induced map. The representation we are interested in is $\rho \sigma \colon G \to B(\calK)$, which by construction factors through a finite $\Cstar$-algebra. To reach a contradiction to \cref{fixedpoints:finite}, we need to exhibit a vector $\eta \in \calK^{\rho \sigma(H)} \setminus \calK^{\rho \sigma(\Gamma)}$. We will find $\eta$ by building a cocycle and exploiting property $(T)$.

\medskip

\emph{A cocycle.} For each $n$ and $g \in \Gamma$ define $c_n(g) \coloneqq \pi_n(g \cdot a_0) - \pi_n(a_0) \in \calK_n$, and let $c \colon \Gamma \to \calK$ be the corresponding internal map, which is well-defined by \eqref{eq:bound:r}. We want to show that $c$ is a cocycle for the representation $\rho\sigma$. Note that by \eqref{eq:asym:mult}:
\begin{equation}
\label{eq:cocycle:op}
\| c_n(hg) - (c_n(h) + \sigma_n(h) c_n(g))\| \to 0 \text{ for all } g, h \in \Gamma.
\end{equation}
To show that $c$ is a cocycle, we need to show that this also holds for $\| \cdot \|_{\calK_n}$. Let $D_n(h, g)$ be the term inside the $\| \cdot \|$.

Let $g \in \Gamma$ and let $r$ be given by an expression like \eqref{eq:expression}. If $u$ and $v$ are commuting self-adjoint unitaries, then $\Tr|u-v| = \frac12 \| u - v \|_{HS}^2$. Applying this to $u = \pi_n(hg \cdot a_0)$ and $v = \pi_n(h \cdot a_0)$, using \eqref{eq:diff:equal} and \eqref{eq:bound:r} we obtain, for sufficiently large $n$:
\begin{align*}
    \Tr|c_n(hg) - c_n(h)| &= \Tr|\pi_n(hg \cdot a_0) - \pi_n(h \cdot a_0)| = \frac12 \| \pi_n(hg \cdot a_0) - \pi_n(h \cdot a_0) \|_{HS}^2 \\
    &= \frac12 \| \pi_n(g \cdot a_0) - \pi_n(a_0) \|_{HS}^2 \leq \frac12 r^2 k_n^2.
\end{align*}
Applying it to $u = \pi_n(g \cdot a_0)$ and $v = \pi_n(a_0)$ we similarly obtain $\Tr|\sigma_n(h) c_n(g)| \leq \frac12 r^2 k_n^2$. Now we estimate:
\begin{align*}
\| D_n(h, g) \|_{\calK}^2 &= \frac{1}{k_n^2} \Tr|D_n(h, g)|^2 \leq \frac{1}{k_n^2} \| D_n(h, g) \| \Tr|D_n(h, g)| \\
&\leq \frac{1}{k_n^2} \| D_n(h, g) \| (\Tr|c_n(hg) - c_n(h)| + \Tr|\sigma_n(h)c_n(g)|) \leq \| D_n(h, g)\| r^2,
\end{align*}
and this goes to $0$ by \eqref{eq:cocycle:op}.

\medskip

\emph{Conclusion.} Since $\Gamma$ has property $(T)$, the cocycle $c$ is a coboundary, that is, there exists $\eta \in \calK$ such that $c(g) = \rho \sigma(g) \eta - \eta$. Now $H$ fixes $a_0$, so $c$ vanishes identically on $H$, and hence $\eta \in \calK^{\rho \sigma(H)}$. However, there exists some $i$ such that the set of $n$ satisfying $k_n = \|\pi_n(\gamma_i \cdot a_0) - \pi_n(a_0)\|_{HS}$ belongs to $\omega$, and so for this $i$ it holds $\|c(\gamma_i)\|_{\calK} = 1$, and in particular $\eta \notin \calK^{\rho\sigma(\gamma_i)}$. This contradicts \cref{fixedpoints:finite}, and concludes the proof.
\end{proof}

\subsection{Non-MF Brin--Thompson groups}

To wrap up, we need a way to embed such groups into Brin--Thompson groups. There is a theory of \emph{demonstrable subgroups} of Thompson groups \cite{demonstrable} which shows that wreath products are hard to embed into $V$, for acting groups more complicated than $\Z$. We have found similar difficulties trying to embed the wreath products of \cite{caleb} into higher Brin--Thompson group, which is why we needed to strengthen \cref{criterion:MF} to \cref{criterion:MF:new}.

\begin{theorem}
\label{thm:MF}
If $nV$ has property $(T)$, then $(n+1)V$ is not MF.
\end{theorem}

\begin{proof}
We follow \cref{not:Uw}, and we further let $H \coloneqq t\Gamma t^{-1}$, which is the copy of $nV$ localised at $U_{00}$. We view $\frakC^{n+1}$ as the product $\frakC^n \times \frakC$, and let $G$ act on $\frakC^{n+1}$ by acting on the base and leaving the height invariant, which realises $G < (n+1)V$. Let $\sigma$ the standard involution of $\frakC$ exchanging $0\frakC$ and $1\frakC$. For a dyadic rectangle $K \subset \frakC^n$, define the element $a_K \in (n+1)V$ by:
\[a_K(x, t) \coloneqq
\begin{cases}
(x, \sigma(t)) \text{ if } x \in K; \\
(x, t) \text{ otherwise}.
\end{cases}\]
The $a_K$ are pairwise commuting involutions: the diagram below shows the effect of $a_K a_L = a_L a_K$ in the non-trivial case in which $K$ and $L$ overlap.

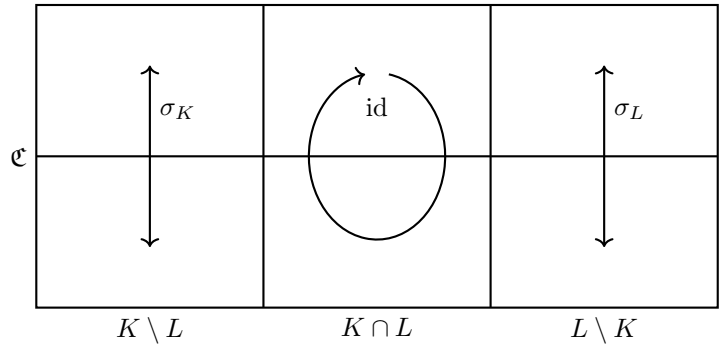
\begin{center}
\begin{tikzpicture}[thick]
  \draw (0,0) rectangle (9,4);
  \draw (3,0) -- (3,4);
  \draw (6,0) -- (6,4);
  \draw (0,2) -- (9,2);

  \node[below] at (1.5,0) {$K \setminus L$};
  \node[below] at (4.5,0) {$K \cap L$};
  \node[below] at (7.5,0) {$L \setminus K$};

  \node[left] at (0,2) {$\mathfrak{C}$};

  \draw[<->] (1.5,3.2) -- (1.5,0.8)
    node[pos=0.25,right] {$\sigma_K$};

  \draw[<->] (7.5,3.2) -- (7.5,0.8)
    node[pos=0.25,right] {$\sigma_L$};

  \draw[->] ({4.5+0.9*cos(80)},{2+1.1*sin(80)})
    arc[start angle=80, end angle=-260, x radius=0.9, y radius=1.1];
  \node[anchor=south] at (4.5,2.4) {id};
\end{tikzpicture}
\end{center}

It follows that the group
\[A \coloneqq \langle a_K : K \subset \frakC^n \text{ a dyadic rectangle}\rangle\]
is a (countable) Boolean group. Moreover, it intersects $G$ trivially, and it is normalised by $G$, because $g a_K g^{-1} = a_{gK}$. Therefore together they generate a semidirect product $W \coloneqq G \ltimes A$. Finally, $a_{U_{01}}$ is fixed by $H$, and not by $\Gamma$; for example, there exists $\gamma \in \Gamma$ such that $\gamma \cdot a_{U_{01}} = a_{U_{00}}$. We can therefore apply \cref{criterion:MF:new} and conclude that $W$, and hence $(n+1)V$, is not MF.
\end{proof}

Together with \cref{thm:main-intro}, this concludes the proof of \cref{main:inapprox}.

\subsection{Inapproximability for the Schatten $p$-norm}

\cref{schatten} will be a direct consequence of \cref{thm:MF} thanks to \cite{1:approx}, therefore we refer the reader to \cite{1:approx} for the relevant notions of approximability and stability.

\begin{proof}[Proof of \cref{schatten}]
    For $1 \leq p \leq \infty$ let $\Gp$ denote the family of finite-dimensional unitary groups endowed with the unnormalised Schatten $p$-norm, for $p = \infty$ this is the operator norm. By \cite[Theorem 1.3 and Remark 3.1]{1:approx}, if a group $G$ contains an involution that is in the kernel of every finite quotient, and $G$ is $\Gq$-stable for some $1 \leq q \leq \infty$, then $G$ is not $\Gp$-approximable for any $1 \leq p \leq q$. Now $nV$, for any $n$, has no non-trivial finite quotients and plenty of involutions. So we only need to show that, for large $n$, the group $nV$ is $\Gq$-stable, for $q = \infty$. Indeed, if $nV$ is not MF (which holds for sufficiently large $n$ by \cref{thm:MF}), then every asymptotic homomorphism to $\mathcal{G}_\infty$ has non-trivial kernel, and because $nV$ is simple, it must be trivial, so $nV$ is vacuously stable.
\end{proof}

This is one of the first naturally occurring examples of a group that is not Schatten $p$-approximable for any $1 \leq p < \infty$. Other such examples can be obtained from finite central extensions of lattices in higher rank Lie groups thanks to the cohomology vanishing criterion from \cite{2:approx, p:approx} and the vanishing results from \cite{B3S}. As for more ad-hoc examples, the proof technique above applies to any simple group containing an involution and a non-MF subgroup, e.g. a simple group containing one of the groups from \cite{caleb}. Even before these developments, the existence, for every $1 \leq p < \infty$, of a finitely presented group that is not Schatten $p$-approximable \cite{2:approx, p:approx, 1:approx} implies that a universal finitely presented group \cite{higman} fails all these approximation properties at once.

\begingroup
\raggedright
\microtypesetup{expansion=false}
\bibliographystyle{amsalpha}
\bibliography{Literature}
\endgroup

\end{document}